\documentclass[11pt]{article}
\usepackage[utf8]{inputenc}
\usepackage{fullpage}

\usepackage{comment}
\usepackage{amssymb}
\usepackage{amsthm}
\usepackage{amsmath}
\usepackage{tikz}

\usepackage{hyperref}

\usepackage{nicematrix}

\newtheorem{thm}{Theorem}[section]
\newtheorem{prop}[thm]{Proposition}

\newtheorem{cor}[thm]{Corollary}
\newtheorem{rmk}[thm]{Remark}
\newtheorem{conj}[thm]{Conjecture}
\newtheorem{ex}[thm]{Example}

\newcommand{\Aut}{\textup{Aut}}

\title{Symmetric  polyhedral scenes and parallel redrawings}
\author{
Signe Lundqvist\thanks{Department of Computer Science, Katholieke Universiteit Leuven, Belgium}
\and
Bernd Schulze\thanks{School of Mathematical Sciences, Lancaster University, UK}
\and
Klara Stokes\thanks{Department of Mathematics and Mathematical Statistics, Ume{\aa} University, Sweden}
}
\date{}

\begin{document}

\maketitle
\begin{abstract}
Liftings and parallel redrawings are classical topics in applied discrete geometry, concerned respectively with vertically lifting a $(d-1)$-picture -- a realisation of a vertex-hyperplane incidence geometry in $\mathbb{R}^{d-1}$ -- to a $d$-dimensional polyhedral scene, and with redrawing a $d$-picture (or hyperplane arrangement) within $\mathbb{R}^d$ while preserving prescribed hyperplane normals. In this paper, we develop a unified framework for the analysis of forced-symmetric liftings and forced-symmetric parallel redrawings. In particular, we show that the classical duality for these theories also appears in the forced-symmetric setting.
We establish the orbit lifting matrix and the orbit concurrence geometry matrix, and show that they are the appropriate symmetry-adapted analogues of the standard lifting matrix and concurrence geometry (or parallel redrawing) matrix. We also make explicit the relationship with the orbit version of Whiteley's parallel design matrix for graphs. %Their kernels characterise precisely the spaces of symmetric liftings and symmetric parallel redrawings. 
Using these tools, we derive necessary conditions for symmetric pictures to be forced-symmetric flat, and for symmetric hyperplane arrangements to be forced-symmetric robust, expressed as sparsity counts on the group-labelled quotient graphs associated with the symmetric incidence geometries. Finally,  we discuss conjectures regarding the sufficiency of these conditions for generic configurations. %We also provide an example, which illustrates that in the generic case (modulo symmetry), these necessary conditions are not sufficient in the lifting setting, and we offer  conjectures on when these conditions might be sufficient.
\end{abstract}


\section{Introduction}

\subsection{Background}
Consider a (combinatorial) incidence geometry consisting of a set of points, a set of hyperplanes, and an incidence relation between them. A realisation of this incidence geometry in $(d-1)$-space, where the points are mapped to locations and the hyperplanes are represented by regions or faces, is called a \emph{$(d-1)$-picture}. One may then ask whether such a $(d-1)$-picture can be lifted vertically to another  realisation of the same incidence geometry in $d$-space, called a  \emph{$d$-scene}, in which the combinatorial hyperplanes become actual affine hyperplanes that remain flat, and  such that not all hyperplanes coincide.

The problem of lifting $2$-pictures to $3$-scenes has long been a central topic in geometric reasoning, computer vision, and artificial intelligence. Reconstructing polyhedral environments or objects from monocular input is a fundamental challenge in scene analysis and robotic perception. Early foundational work by Mackworth~\cite{Mac} and Huffman~\cite{Huff} provided necessary conditions for the realisability of $2$-pictures as $3$-scenes. The geometric method of reciprocal diagrams used by them had, in fact, already appeared in the work of Maxwell and Cremona as a graphical tool for analysing equilibrium stresses in bar-joint frameworks, revealing a strong connection to geometric rigidity theory \cite{Maxwell01041864,baker2025geometry}. See also \cite{WW84cor}.

Further progress was made by Sugihara~\cite{Sug841,Sug842}, who gave both a necessary and sufficient condition for a general picture to be non-liftable to a polyhedron, or \emph{flat},    using linear programming. Subsequently, necessary and sufficient conditions were obtained by Crapo and Whiteley~\cite{Crapo,CW93,RosThom} using invariant-theoretic and projective-geometric methods. Liftability can be characterised by analysing the independence of the rows in a \emph{lifting matrix}, giving necessary conditions for flatness in terms of sparsity counts involving the elements of the incidence geometry. Following a conjecture of Sugihara, Whiteley later showed that these conditions are also sufficient for generic pictures to be flat in all dimensions \cite{WW89}.

The dual problem to liftability is that of \emph{parallel redrawings} of hyperplane arrangements. Here one asks whether for a given  realisation of  a vertex-hyperplane incidence geometry in $\mathbb{R}^{d}$, there exists another  realisation of this incidence geometry in $\mathbb{R}^{d}$  that preserves all prescribed hyperplane normals, and is non-trivial, i.e. not obtained from translations or dilation.  Whiteley gave a complete characterisation for the non-existence of such non-trivial parallel redrawings for generic hyperplane arrangements in all dimensions in \cite{WW89}. We note that for the special case when the incidence geometry is a simple graph, the existence of a non-trivial parallel redrawing of a $2$-dimensional bar-joint framework is equivalent to the existence of a non-trivial infinitesimal motion, a fact long used as a graphical tool in engineering.

Motivated by practical applications, where symmetry in structures is ubiquitous, the role of symmetry has recently become increasingly important in rigidity and geometric realisation theory (see, e.g., \cite{connellyguest,bernd2017sym}). The theory distinguishes between two settings: \emph{incidental symmetry}, where the original symmetry of a picture may be lost in a lifting, parallel redrawing, or (infinitesimal) motion; and \emph{forced symmetry}, where the original symmetry must be preserved.
See Figure~\ref{fig:prismlift} for an example of a generically flat $2$-picture with reflection symmetry that can be lifted to a non-trivial $3$-scene that is also mirror-symmetric.

\begin{figure}[htp]
    \centering
    \includegraphics[width=0.28\linewidth]{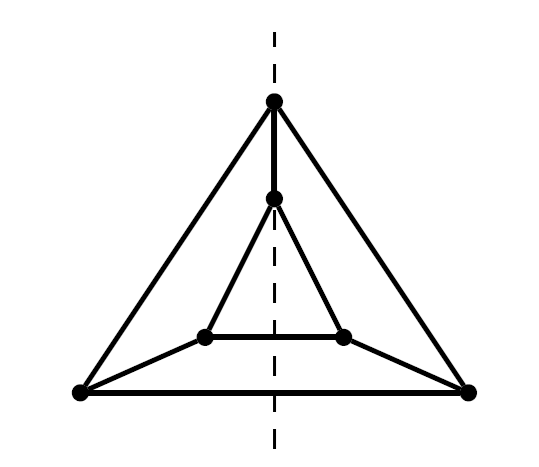}
    \hspace{0.5cm}
    \includegraphics[width=0.32\linewidth]{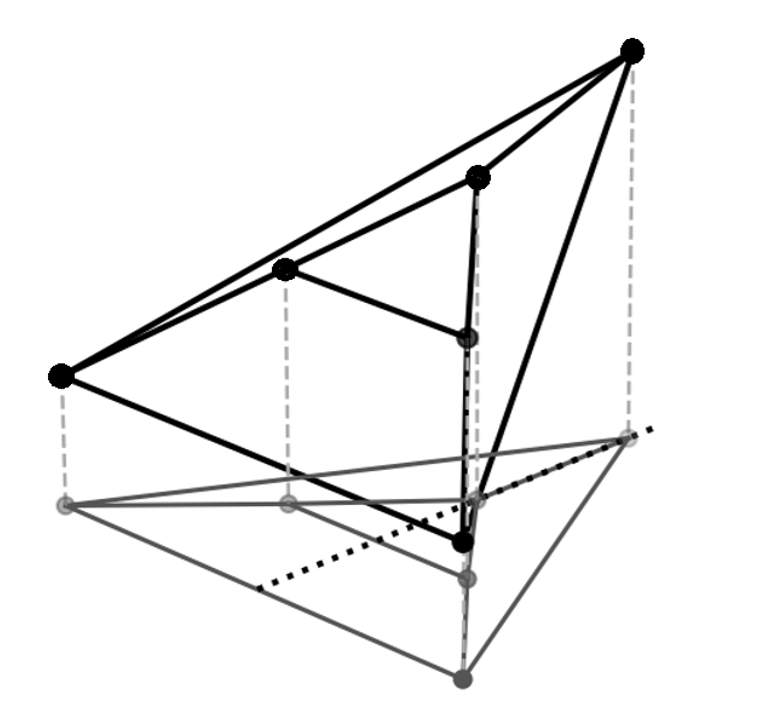}
      \caption{A $2$-picture with vertical reflection symmetry, where the hyperplanes correspond to the faces of the planar graph. Such a reflection-symmetric $2$-picture admits a non-trivial lifting to a $3$-scene that also preserves the reflection symmetry, as shown on the right.}
    \label{fig:prismlift}
\end{figure}

The current state of the art for liftings and parallel redrawings of symmetric pictures and hyperplane arrangements can be summarised as follows.

In \cite{KS17} Kaszanitzky and Schulze derived necessary conditions for symmetric $(d-1)$-pictures to be (minimally) flat using a block-decomposition of the lifting matrix corresponding to irreducible representations of the picture's symmetry group. For $3$-fold rotational symmetry in the plane, these conditions, together with the standard sparsity conditions, were also shown to be sufficient under suitable genericity assumptions \cite{KS18}; these results concern the \emph{incidental symmetry} setting, where the original symmetry of the picture may be lost in a lifting. By contrast, the problem of \emph{forced symmetry}, where the lifting must preserve the original symmetry, remains largely unexplored. Necessary conditions for $2$-dimensional pictures to be \emph{forced-symmetric flat} in this sense were given in \cite{JKS}, but a general theory is still lacking.

In \cite{tanigawamatroids} Tanigawa established complete combinatorial characterisations for  realisations of \emph{simple graphs} in $\mathbb{R}^{d}$ (i.e., $d$-dimensional bar-joint frameworks) to admit no non-trivial symmetric parallel redrawings for arbitrary symmetry groups, under the assumption that the group action is free on the vertices. However, extensions of this result to more general incidence geometries remain open.

\subsection{Our contributions}

In this work, we focus on symmetric pictures where the group action on points is free, which is the standard  setting in the literature. We first introduce the \emph{orbit lifting matrix}, which is the symmetry-adapted version of the standard lifting matrix and serves as the analogue of the well established orbit rigidity matrix in rigidity theory \cite{sw2011} (see Section~\ref{sec:sym_scenes}). We also define the \emph{orbit concurrence geometry matrix} (see Sections~\ref{sec:dualsym} and \ref{sec:pardrgen}). 
We then establish that the classical duality between liftings and parallel redrawings extends to the forced symmetric setting.
In the planar case, we show that the orbit concurrence geometry matrix generalises Tanigawa's orbit matrix for parallel redrawings of graphs. In particular, if the incidence geometry is a graph, the orbit concurrence geometry matrix reduces to Tanigawa's orbit matrix (see Section~\ref{sec:robustness}). Both matrices can be viewed as the blocks corresponding to the trivial representation in the block-decomposed lifting and concurrence geometry matrices, and their kernels describe precisely the spaces of symmetric liftings and symmetric parallel redrawings, respectively (see  Theorems~\ref{Thm:scene_orbit_kernel} and \ref{thm:sym_kernel_O(d)}). 

Using these new types of orbit matrices, we derive necessary conditions for symmetric pictures to be \emph{forced-symmetric flat}, that is, to admit no non-trivial vertical symmetry-preserving lifting to a polyhedral scene, and for symmetric hyperplane arrangements to be \emph{forced-symmetric robust}, that is, to admit no symmetry-preserving non-trivial parallel redrawing (see Corollaries~\ref{cor:sym_flat_nec} and \ref{cor:sym_robust_nec}). These conditions take the form of sparsity counts on the group-labelled quotient graphs of the symmetric bipartite incidence graphs describing the incidence geometries under the action of the symmetry group. 

Finally, we offer some conjectures regarding the sufficiency of the necessary conditions when the realisations are generic (modulo the symmetry) and discuss some further directions for future work.

%discuss sufficient conditions in the case when the realisation is generic (modulo the symmetry), and show that  the necessary conditions are in general not sufficient for a picture to be forced-symmetric flat. In contrast, for parallel redrawings,  the conditions may be generically sufficient. These observations highlight a structural source of row dependencies in the lifting setting and suggest a conjecture on when the necessary sparsity conditions are also sufficient in the symmetry-generic setting. 

\subsection{Structure of paper}

In Section \ref{sec:preliminaries} we introduce pictures, scenes, liftings, hyperplane arrangements and parallel redrawings, which are the main concepts of the article, and we summarise key results from the corresponding non-symmetric theories.  In Section~\ref{sec:sym} we define the orbit lifting matrix and in Section \ref{sec:symmetric_parallel} the orbit concurrence
geometry matrix, and use those matrices to establish necessary conditions for symmetry-forced flatness and robustness, respectively. In Section~\ref{sec:symmetric_parallel} we also estalish the duality between the forced-symmetric lifting and parallel redrawing problems. In Section~\ref{sec:robustness} we  establish the symmetry-adapted analogue of a classical relationship (due to W. Whiteley) between the parallel design matrix of a bar-joint framework and the concurrence geometry matrix of the associated hyperplane arrangement in dimension $2$. In Section~\ref{sec:conjectures} we discuss conjectures on sufficiency of the necessary conditions.

\section{Preliminaries}
\label{sec:preliminaries}

\subsection{Incidence geometries}

An \textit{incidence geometry} (of rank two) is a triple of sets $S=(P,L,I)$, where $I$ is a subset of $P \times L$. We call the elements of $P$ \textit{points} and the elements of $L$ \textit{hyperplanes} (sometimes denoted $H$ in geometric contexts, but we use $L$ for consistency with classical incidence geometry notation). The set $I$ is a set of \textit{incidences} between the points and hyperplanes. Incidence geometries of rank two are hypergraphs, where the hyperedges are the sets of points incident to a common hyperplane. If every hyperplane is incident to exactly two points, then the incidence geometry is a \emph{graph}.

The \emph{incidence graph} of an incidence geometry is a bipartite graph with vertex
set $V=P\cup L$, and an edge $(p,\ell)$ if and only if $(p,\ell)\in I$. An incidence geometry
is said to be \emph{connected} if its incidence graph is connected. Throughout, we will only consider connected incidence geometries.

%\subsection{Projective geometry}
%\label{sec:projective_geometry_prelims}

\subsection{Pictures and scenes}
\label{sec:pic_scenes}

A \textit{$(d-1)$-picture} $(S,x)$ consists  of an incidence geometry $S=(P,L,I)$ and an assignment $x: P \to \mathbb{R}^{d-1}$. A \textit{$d$-scene} lifting a $(d-1)$-picture $(S,x)$ 
is an assignment of a point coordinate $y: P \to \mathbb{R}$ to every $p \in P$, an assignment  $n: L \to \mathbb{R}^{d-1}$ defining the normals $(n(\ell),1)$ of the hyperplanes, as well as an affine shift 
$h: L \to \mathbb{R}$ to each element of $L$ such that for each incidence $(p, \ell)$ we have
\begin{equation}
    n(\ell)x(p)+ y(p)+ h(\ell)=0.
    \label{scene_eq}
\end{equation}

When considering $d$-scenes we will always assume that each hyperplane of the incidence geometry is incident to at least $d$ points. Finding the space of $d$-scenes lifting a $(d-1)$-picture $S(x)$ of an incidence geometry $S$ amounts to solving a system of $|I|$ equations of the form (\ref{scene_eq}), where $n(\ell)$, $h(\ell)$ and $y(p)$ are variables and $x(p)$ is fixed. We denote the coefficient matrix of this system of equations by $M^{d}(S, x)$.

The row of $M^{d}(S,x)$ corresponding to the incidence $(p, \ell)$ is given by
 \setcounter{MaxMatrixCols}{20}
  $$\begin{bNiceMatrix}[first-row,first-col]
         &&&&\ell&&&&&p&&&&\\
		  &0&\ldots&0&x(p)^T \quad 1&&0&\ldots&0&1&0&\ldots&0
        \end{bNiceMatrix}$$

A $d$-scene is \textit{trivial} if $n(\ell)=n(\ell')$ and $h(\ell)=h(\ell')$ for all $\ell, \ell' \in L$, meaning that all hyperplanes representing elements of $L$ lie in the same hyperplane. A picture is \textit{flat} if all scenes lifting the picture are trivial. It is \emph{minimally flat} if it is flat and the removal of any incidence yields a non-flat picture. A $d$-scene is \textit{sharp} if all hyperplanes representing elements of $L$ are distinct. 

We can also consider the points $x(p)$ as points in $\mathbb{RP}^{d-1}$, denoted by homogeneous coordinates $[x_1:\ldots:x_{d}]$. Throughout, we make the assumption that $x_{d}$ is non-zero, which means that the point is not at infinity, and use the representative $(x_1,\ldots,x_{d-1},1)$ for the point in $\mathbb{RP}^{d-1}$.

Similarly, we can consider the hyperplanes as hyperplanes in $\mathbb{RP}^{d}$, and denote a hyperplane $a_1x_1+\cdots+a_{d+1}x_{d+1}=0$ by its homogeneous coordinates $[a_1:\ldots:a_{d+1}]$. Throughout, we will assume that $a_{d} \neq 0$, which means that the hyperplane is not parallel to the $d:$th coordinate axis of $\mathbb{R}^d$.

We will switch freely between the vector representation  and projective viewpoint in this paper, as convenient.

\subsection{Hyperplane arrangements and parallel redrawings}
\label{sec:hyperplanes_parallel}

Given an incidence geometry $S=(P,L,I)$, let $n$ be an assignment $n: L \to \mathbb{R}^{d-1}$. We can then consider the \emph{hyperplane arrangement} $(S,n)$ where $(n(\ell),1)$ is the hyperplane normal assigned to the line $\ell \in L$. Assuming that the last coordinate of each normal is $1$ means that the hyperplanes representing elements of $L$ are not parallel to the $d$-th coordinate axis of $\mathbb{R}^d$.

A \textit{parallel redrawing} of $(S,n)$  is an assignment of an affine shift $h \in \mathbb{R}$ to each element of $L$, together with an assignment of a point $(x,y)$ in $\mathbb{R}^d$, where $x \in \mathbb{R}^{d-1}$ and $y \in \mathbb{R}$, to each element of $P$ such that if $(p, \ell) \in I$, then the point $(x(p),y(p))$ in $\mathbb{R}^d$ representing $p$ lies in the hyperplane $n(\ell)x+y+h(\ell)$. Finding the space of parallel redrawings of a hyperplane arrangement $(S,n)$ amounts to solving an equation of the form 

\begin{equation}
    n(\ell)x(p)+y(p)+h(\ell)=0
    \label{parallel_eq}
\end{equation}
for each incidence $(p, \ell)$, where $x(p)$, $y(p)$ and $h(\ell)$ are variables and $n(\ell)$ is fixed. When considering parallel redrawings in dimension $d$, we will assume that each point is incident to at least $d$ lines. 

We denote the coefficient matrix of this system of equations by $M^{d*}(S, n)$. This matrix is also known as the \emph{concurrence geometry matrix} (in reference to Crapo's term of a concurrence geometry \cite{CrapoCon}).

The row of $M^{d*}(S,n)$ corresponding to the incidence $(p, \ell)$ is given by

 \setcounter{MaxMatrixCols}{20}
  $$\begin{bNiceMatrix}[first-row,first-col]
         &&&&p&&&&\ell&&&&\\
		  &0&\ldots&0&n(\ell) \quad 1&0&\ldots&0&1&0&\ldots&0
        \end{bNiceMatrix}$$

We say that a parallel redrawing is \textit{degenerate} if all points are assigned the same coordinates. All incidence hyperplane arrangements $(S,n)$ have degenerate parallel redrawings. If a hyperplane arrangement has some non-degenerate parallel redrawing, i.e. a parallel redrawing where at least two points are given distinct coordinates, then any dilation is also another parallel redrawing of $(S,n)$. We say that a hyperplane arrangement $(S,n)$ is \textit{robust} if the only parallel redrawings of $(S,n)$ are \emph{trivial}, i.e., translations and dilations. 

For an incidence geometry $S=(P,L,I)$ we can define its \textit{dual} $S^*=(L,P,I)$, where the point set and the hyperplane set are switched. 
In projective duality, points correspond to hyperplanes, and a hyperplane constrained by an assigned normal corresponds to a point constrained by one projection. Thus, liftings of the primal structure correspond naturally to parallel redrawings of the dual structure, in the sense that they provide complementary geometric descriptions that preserve incidence relations (ee e.g.,\cite{DCGHand}). We will consider this duality in the symmetry-forced setting in Section~\ref{sec:dualsym}.

\subsection{The $k$-plane matroid}
\label{sec:k-plane_matroid}

Let $G=(A \cup B,E)$ be a bipartite graph. The \textit{$k$-plane matroid at $B$} is a matroid on E, where a set $E' \subseteq E$ is independent whenever

$$
|E''| \leq |A(I'')| + k|B(I'')| - k
$$
for all non-empty subsets $E'' \subseteq E'$. 

We say that a picture $(S,x)$ in $\mathbb{R}^{d-1}$ is \textit{generic} if the set of coordinates $\{x(p) \mid p \in P\}$ is algebraically independent over $\mathbb{Q}$. The following is a fundamental result on liftings of pictures, due to W. Whiteley.

\begin{thm}[The picture theorem, Whiteley \cite{WW89}]
    Let $S=(P,L,I)$ be an incidence geometry. The rows of $M^{d}(S,x)$ are independent for all generic  $(d-1)$-pictures $(S,x)$ if and only if the incidence graph $G=(P \cup L, I)$ is independent in the $d$-plane matroid at $L$.
    All generic $(d-1)$-pictures $(S,x)$ are minimally flat  if and only if furthermore $|I|=|P|+d|L|-d$, i.e. if $I$ is a basis of the $d$-plane matroid at $L$.
    \label{thm:pic_thm}
\end{thm}

The dual situation holds for parallel redrawings. We say that a hyperplane arrangement $(S,n)$ in $\mathbb{R}^d$  is \emph{generic} if the set of hyperplane normals $\{n(\ell) \mid \ell \in L\}$ is algebraically independent over $\mathbb{Q}$.

\begin{thm}[Whiteley \cite{WW89}]
    Let $S=(P,L,I)$ be an incidence geometry. The rows of $M^{d*}(S, n)$ are independent for all generic hyperplane arrangements $(S,n)$ if and only if $I$ is independent in the $d$-plane matroid at $P$.
    \label{independence_parallel}
\end{thm}

Whiteley also addressed the question of which incidence geometries have \emph{proper} realisations, i.e. realisations where all points are distinct, for generic choices of hyperplane normals.
%, that is, for ``generic" choices of normals for the faces. 
The answer is given in Theorem \ref{robust} below.

\begin{thm}[Whiteley \cite{WW89}]
    Hyperplane arrangements $(S,n)$ in $\mathbb{R}^d$ with generic normals are proper if and only if 
    $$
    |I'| \leq |L(I')|+d|P(I')|-(d+1)
    $$
    for all subsets $I' \subseteq I$ such that $|I'| \geq 2$. 
    
    Hyperplane arrangements of $S$ with generic normals are minimally robust if and only if furthermore $|I|=|L|+d|P|-(d+1)$.
    \label{robust}
\end{thm}

\begin{rmk}
Note that there is a discrepancy between the dimensions of the spaces of trivial liftings and trivial parallel redrawings: trivial liftings form a $d$-dimensional space, whereas trivial parallel redrawings form a $(d+1)$-dimensional space, due to  the presence of a dilation in the latter. 

The dual of a flat scene is a degenerate hyperplane arrangement where all lines go through a single point. The dilation acts trivially on such hyperplane arrangements, and in this case the kernel of $M^{d*}(S,n)$ only contains the translations. Conversely, the dual of a robust hyperplane arrangement will be a picture that has a non-trivial scene corresponding to the dilation.
\end{rmk}

%\signe{Yes, we can add something like that! Except don't we impose a normalisation condition on the hyperplanes in both the scene analysis and parallel redrawings cases (actually the same one, that the hyperplanes are not vertical)?}\bernd{yes, you are right. I rephrased. But maybe it can be improved? Signe: can you add sentence saying that translations correspond to trivial liftings, but that dilation corresponds to a non-trivial lifting...?}

\section{Symmetry-forced liftings and parallel redrawings}\label{sec:sym}

\subsection{Symmetric incidence geometries}
\label{sec:symincgeo}

An \emph{automorphism} of an incidence geometry $S=(P,L,I)$ is a bijection $\alpha: (P,L) \to (P,L)$ such that $(p,\ell) \in I$ if and only if $(\alpha(p), \alpha(\ell)) \in I$. The set of automorphisms of an incidence geometry forms a group  $\Aut(S)$ under composition. 
Note that an automorphism of an incidence geometry naturally induces a graph automorphism of the incidence graph, which maps the sets of vertices of the incidence graph representing elements of $P$  and $L$ to themselves.

%Let $\Gamma$ be a finite group and let $\phi:  \Gamma \to \Aut(S)$ be a homomorphism. Then $\phi$  defines an action of $\Gamma$ on the incidence graph $G$ of $S$. If $\Gamma$ is a subgroup of $\textrm{Aut}(S)$, then $S$ and the corresponding incidence graph $G$ are said to be \emph{$\Gamma$-symmetric}.

If $\Gamma$ is a subgroup of $\mathrm{Aut}(S)$, then we say that $S$ (and its incidence graph $G$) is $\Gamma$-symmetric, with $\Gamma$ acting naturally on $S$ and on $G$.
Throughout this paper, we will only consider automorphisms of $S$ that act freely on the vertices of its incidence graph. In other words, no vertex of the incidence graph is fixed by any non-trivial element of $\Gamma$.

To study symmetric pictures, scenes and parallel redrawings, it is convenient to consider group-labelled quotient graphs of the corresponding 
symmetric incidence graphs under the action of the group $\Gamma$. These graphs are also known as ``gain graphs". We make this precise below and refer the reader to \cite{jkt} for further details. 

Let $G$ be $\Gamma$-symmetric. The action of $\Gamma$ on $G=(V,E)$ defines a \textit{quotient $\Gamma$-gain graph} $(G/\Gamma, \psi)$ with a vertex for each orbit of vertices under the action of $\Gamma$, and an edge for each orbit of edges under the action of $\Gamma$. To define the \emph{gain function} $\psi: E\to \Gamma$,  we first choose a representative from each vertex orbit. 
Let $u$ and $v$ be representatives of two vertex orbits such that there exists an edge $(u, \gamma \cdot v) \in E(G)$ for some $\gamma \in \Gamma$,  where $\gamma \cdot v$ denotes the image of $v$ under the action of $\gamma$. Then the directed edge $e=(u,v)$ in $(G/\Gamma, \psi)$ is assigned the gain $\psi(e) = \gamma$. 

Since the incidence graph of an incidence geometry is bipartite, edges connect vertices representing points $P$ to vertices representing hyperplanes $L$. We adopt the convention of directing edges in the quotient gain graph consistently with this bipartition: 
\begin{itemize}
\item For liftings to scenes, edges go from hyperplane-orbits to point-orbits. 
    \item  For parallel redrawings, edges go from point-orbits to hyperplane-orbits.  
\end{itemize}
 This convention ensures that the gain assigned to an edge records how the corresponding group element acts on the vertex at the head of the directed edge, leading to a consistent symmetry-adapted matrix formulation in each setting.

We denote the vertex set of the quotient $\Gamma$-gain graph $(G/\Gamma,\psi)$ by $P_\Gamma \cup L_\Gamma$, where $P_\Gamma$ is the set of vertices representing orbits of points, and $L_\Gamma$ is the set of vertices representing the orbits of hyperplanes. We denote its edge set  by $I_\Gamma$.

Conversely, we may start with a finite group $\Gamma$ and a multigraph $G'=(V',E')$ (with no loops), where $V'=P'\cup L'$ and each edge is oriented consistently with the chosen convention (that is, from hyperplane-vertices to point-vertices in the lifting setting, and from point-vertices to hyperplane-vertices in the parallel redrawing setting), and assigned a group label (or gain) via a function $\psi:E'\to\Gamma$,  such that parallel edges are assiged distinct  gains.		
This gives a \emph{$\Gamma$-gain graph}  $(G',\psi)$. An edge $e$ from vertex $v_i$ to vertex $v_j$ with gain $\psi(e)$ is denoted by $e=(v_i,v_j;\psi(e))$.	
        
For a $\Gamma$-gain graph $(G',\psi)$, the \emph{derived (or covering) graph} $G$ is the $\Gamma$-symmetric simple graph that is represented by $(G',\psi)$. More precisely, if $G'=(V',E')$, then $G=(V,E)$, where
\[
V = \{(v,\gamma) \mid v \in V',\ \gamma \in \Gamma\},
\]
and an edge $\{\gamma_i v_i, \gamma_j v_j\}$ is in $E$ if and only if 
\[
(v_i, v_j; \gamma_i^{-1} \gamma_j) \in E'.
\]

%More precisely, if $G'=(V',E')$, then $G=(V, E)$, where  $V=\{(v,\gamma):v\in V',\gamma\in\Gamma\}$ and $\{(v_i,\gamma),(v_j,\gamma\psi(e))\}$ is an edge of $E$ whenever $e=(v_i,v_j)\in E'$. 

Let $W$ be a walk in a gain graph $G$. The \textit{gain of $W$} is defined as $\psi(W)=\psi(e_1)\psi(e_2)\ldots\psi(e_n)$, where $e_1,\ldots,e_n$ are the edges traversed in order along $W$, and the product is taken using the group operation in $\Gamma$. If some edge in the walk is traversed against the direction of the edge, then $\psi(e_i)$ should be replaced by $\psi(e_i)^{-1}$ in the product.

For a subset $X \subseteq E$, let $\pi_1(X,v)$ denote the set of all closed walks in $G$ starting at a vertex $v$ and using only the edges of $X$. If $v \notin V(X)$, we define $\pi_1(X,v)= \emptyset$. The \textit{subgroup generated by $X$ relative to $v$} is defined as $\langle X \rangle_{\psi, v}= \langle \psi(W) \mid W \in \pi_1(X, v) \rangle$. The groups $\langle X \rangle_{\psi, v}$ and $\langle X \rangle_{\psi, u}$ are conjugate if $X$ is connected, and $u$ and $v$ are two vertices in $V(X)$ \cite{jkt}. A subset $X \subseteq E$ is \textit{balanced} if $\langle X \rangle_{\psi,v}$ is trivial for every $v \in V(X)$. Otherwise $X$ is called \textit{unbalanced}.

One may choose different representatives for the vertex orbits in a $\Gamma$-gain graph, which corresponds to what is called a \emph{switching operation} (see \cite{jkt} for details). Such operations do not change the derived graph and preserves the group up to conjugation, and it is routine to show that they also preserve the rank of  the associated orbit matrices, including the orbit lifting matrix and the orbit concurrence geometry matrix.
%\signe{As it is, we don't actually use switching operations anywhere, but we might want to note that the rank of the orbit matrix is independent of switching operations just so our definitions don't seem ambiguous}
%A \emph{switching operation} in a (quotient) $\Gamma$-gain graph $(G/ \Gamma, \psi)$ amounts to choosing a different representative of the corresponding vertex orbit. Specifically, if $(G/\Gamma, \psi)$ is a (quotient) $\Gamma$-gain graph, then the gain function obtained from $\psi$ by a switching operation at $v$ with $\tau(\gamma)$ is the gain function that is obtained by choosing the representative $\tau(\gamma) v$ for the orbit of vertices containing $v$.
%Specifically, a \textit{switching operation at $v$ with $\tau(\gamma)$} is a change to the gain function $\psi$ by
%\[  \psi'=
%  \begin{cases}
%     \tau(\gamma) \psi(e) \tau(\gamma)^{-1} & \quad \text{if } e \text{ is a loop at } v,\\
%      \tau(\gamma) \psi(e) &\quad \text{if } e \text{ is directed from $v$ and not a loop}\\
%      \psi(e) \tau(\gamma)^{-1} & \quad \text{if $e$ is directed to $v$ and not a loop}\\
%      \psi(e) & \quad \text{if $e$ is not incident to $v$}.
%  \end{cases}
%\]
Two gain functions on the same underlying graph are said to be \textit{equivalent} if one can be obtained from the other by a sequence of switching operations.

%\signe{Do we need these propositions? We don't cite them anywhere.}
%\begin{prop} \cite{jkt}
%Let $(G, \psi)$ be a gain graph, and let $\psi'$ be a gain function equivalent to $\psi$. Then $\langle X \rangle_{\psi, v}$ and $\langle X \rangle_{\psi', v}$ are conjugate for any set $X \subseteq E$ and $v \in V(X)$. 
%\label{eqv_gains}
%\end{prop}

%\begin{prop} \cite{jkt}
%    Let $(G, \psi)$ be a gain graph. Then, for any forest $F \subseteq E$, there is a gain function $\psi'$ which is equivalent to $\psi$ such that $\psi'(e)=id$ for all $e \in F$. 
%\label{forest_gains}
%\end{prop}

\subsection{Symmetric scenes and the orbit lifting matrix} \label{sec:sym_scenes}

In this section, we define the orbit lifting matrix, whose kernel describes the space of symmetric $d$-scenes lifting a symmetric $(d-1)$-picture, and we derive necessary conditions for independence in its row matroid. To introduce symmetric liftings and the orbit lifting matrix, we first specify the class of symmetry groups under consideration. 

A symmetric $d$-scene is a scene in $\mathbb{R}^d$ that is invariant under the action of a finite subgroup of $O(d)$, meaning that the group acts linearly by rotations or reflections and preserves the scene. Before presenting the general formal definition, we give a few illustrative examples to highlight subtleties that arise in symmetric liftings. In this work, we focus on symmetric scenes that arise as liftings of symmetric pictures in $\mathbb{R}^{d-1}$, so that the lifted scene preserves the original symmetry of the picture.

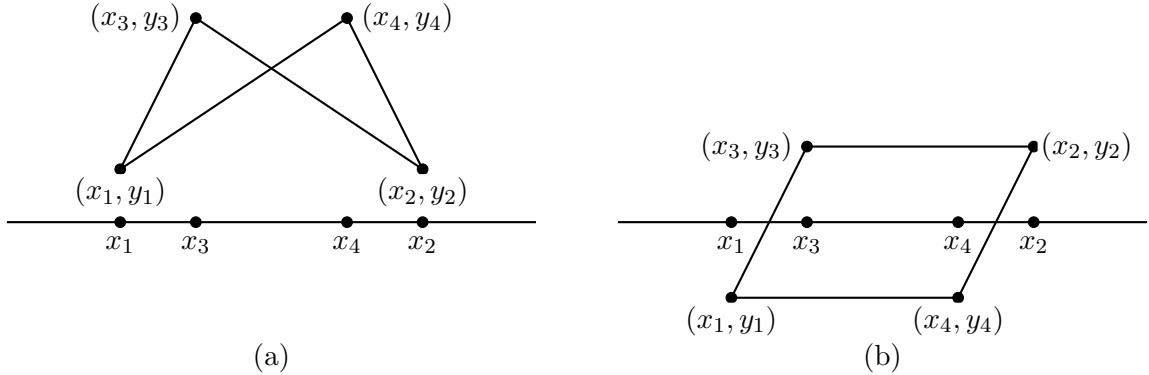
\begin{figure}[h]
    \centering
    \begin{tikzpicture}
        \filldraw[black] (-2,-1) circle (2pt);
        \filldraw[black] (2,-1) circle (2pt);
        \filldraw[black] (-1,1) circle (2pt);
        \filldraw[black] (1,1) circle (2pt);
        \node[fill=white, draw=white, inner sep=1pt, anchor=center] at (-2,-1.3) {$(x_1,y_1)$};
         \node[fill=white, draw=white, inner sep=1pt, anchor=center] at (2,-1.3) {$(x_2,y_2)$};
          \node[fill=white, draw=white, inner sep=1pt, anchor=center] at (-1.8,1) {$(x_3,y_3)$};
           \node[fill=white, draw=white, inner sep=1pt, anchor=center] at (1.8,1) {$(x_4,y_4)$};

        \filldraw[black] (-2,-1.7) circle (2pt);
        \filldraw[black] (2,-1.7) circle (2pt);
        \filldraw[black] (-1,-1.7) circle (2pt);
        \filldraw[black] (1,-1.7) circle (2pt);
\node[fill=white, draw=white, inner sep=1pt, anchor=center] at (-2,-2) {$x_1$};
         \node[fill=white, draw=white, inner sep=1pt, anchor=center] at (2,-2) {$x_2$};
          \node[fill=white, draw=white, inner sep=1pt, anchor=center] at (-1,-2) {$x_3$};
           \node[fill=white, draw=white, inner sep=1pt, anchor=center] at (1,-2) {$x_4$};

        \draw[thick] (-3.5,-1.7)--(3.5,-1.7);
        \draw[thick] (-1,1)--(-2,-1);
        \draw[thick] (-1,1)--(2,-1);
        \draw[thick] (1,1)--(2,-1);
        \draw[thick] (1,1)--(-2,-1);
        \node[fill=white, draw=white, inner sep=1pt, anchor=center] at (0,-3.5) {(a)};
    \end{tikzpicture}
    \hspace{0.8cm}
     \begin{tikzpicture}
        %\filldraw[black] (0,3) circle (2pt);
        \filldraw[black] (-2,-2.7) circle (2pt);
        \filldraw[black] (2,-0.7) circle (2pt);
        \filldraw[black] (-1,-0.7) circle (2pt);
        \filldraw[black] (1,-2.7) circle (2pt);
\node[fill=white, draw=white, inner sep=1pt, anchor=center] at (-2,-3) {$(x_1,y_1)$};
         \node[fill=white, draw=white, inner sep=1pt, anchor=center] at (1,-3) {$(x_4,y_4)$};
          \node[fill=white, draw=white, inner sep=1pt, anchor=center] at (-1.8,-0.7) {$(x_3,y_3)$};
           \node[fill=white, draw=white, inner sep=1pt, anchor=center] at (2.7,-0.7) {$(x_2,y_2)$};

        \filldraw[black] (-2,-1.7) circle (2pt);
        \filldraw[black] (2,-1.7) circle (2pt);
        \filldraw[black] (-1,-1.7) circle (2pt);
        \filldraw[black] (1,-1.7) circle (2pt);
\node[fill=white, draw=white, inner sep=1pt, anchor=center] at (-2,-2) {$x_1$};
         \node[fill=white, draw=white, inner sep=1pt, anchor=center] at (2,-2) {$x_2$};
          \node[fill=white, draw=white, inner sep=1pt, anchor=center] at (-1,-2) {$x_3$};
           \node[fill=white, draw=white, inner sep=1pt, anchor=center] at (1,-2) {$x_4$};
           
        \draw[thick] (-3.5,-1.7)--(3.5,-1.7);
        \draw[thick] (-1,-0.7)--(-2,-2.7);
        \draw[thick] (1,-2.7)--(2,-0.7);
        \draw[thick] (-1,-0.7)--(2,-0.7);
        \draw[thick] (1,-2.7)--(-2,-2.7);
        \node[fill=white, draw=white, inner sep=1pt, anchor=center] at (0,-3.5) {(b)};
    \end{tikzpicture}
    \caption{(a) A reflectionally symmetric $1$-picture (with combinatorial hyperplanes given by the four pairs of points $13, 14, 23, 24$) that lifts to a reflectionally symmetric $2$-scene. (b) The same $1$-picture also lifts to a half-turn-symmetric $2$-scene.}
    \label{fig:ref_sym_scene}
\end{figure}

\begin{ex}
\label{ex:C2-scenes}
Consider the $1$-picture shown in Figure \ref{fig:ref_sym_scene}. 
This $1$-picture is invariant under the action on $\mathbb{R}$ given by multiplication by $-1$, 
that is, it is symmetric with respect to reflection about the origin. 
 As illustrated in Figure \ref{fig:ref_sym_scene} (a), this $1$-picture admits a lifting to a reflectionally symmetric $2$-scene, i.e a scene in $\mathbb{R}^2$ that is invariant under the action on $\mathbb{R}^2$ given by multiplication by the matrix  $\operatorname{diag}(-1,1)$. 
 %$\begin{pmatrix}
  %  -1 & 0\\
 %   0 & 1\\
%\end{pmatrix}$. 

Note that the same $1$-picture also lifts to a \emph{half-turn-symmetric} $2$-scene, i.e. a $2$-scene that is invariant under the action on $\mathbb{R}^2$ given  by multiplication by the matrix $\operatorname{diag}(-1,-1)$
%$\begin{pmatrix}
 %   -1 & 0\\
 %   0 & -1\\
%\end{pmatrix}$
(see Figure \ref{fig:ref_sym_scene}(b)). 
\end{ex}

\begin{rmk}
Suppose that in Figure \ref{fig:ref_sym_scene}  a point is added to the $1$-picture which is fixed under the $\mathbb{Z}_2$-action (hence necessarily the origin). 
In the reflectionally symmetric lifting (Figure~\ref{fig:ref_sym_scene} (a)), this fixed point is realised as a finite point in the $2$-scene: the corresponding diagonals intersect on the axis of reflection.
In contrast, in the half-turn-symmetric lifting (Figure~\ref{fig:ref_sym_scene} (b)), the corresponding lines form opposite sides of a parallelogram and are therefore parallel. Consequently, the $\mathbb{Z}_2$-fixed point in the $1$-picture cannot be realised as a finite point in the affine $2$-scene and appears only projectively as a point at infinity.  %This phenomenon is specific to points that are fixed under the $\mathbb{Z}_2$-action. (how about two coincident points at origin? maybe no need to say more here.
This phenomenon highlights that, in the presence of fixed points, the choice of character can affect whether such points admit an affine realisation in the lifted scene or only a projective one. While extending the orbit matrix formalism to include fixed points is straightforward, we do not pursue this direction further in this paper.
\end{rmk}

\begin{figure}[htp]
    \centering
    \includegraphics[width=0.25\linewidth]{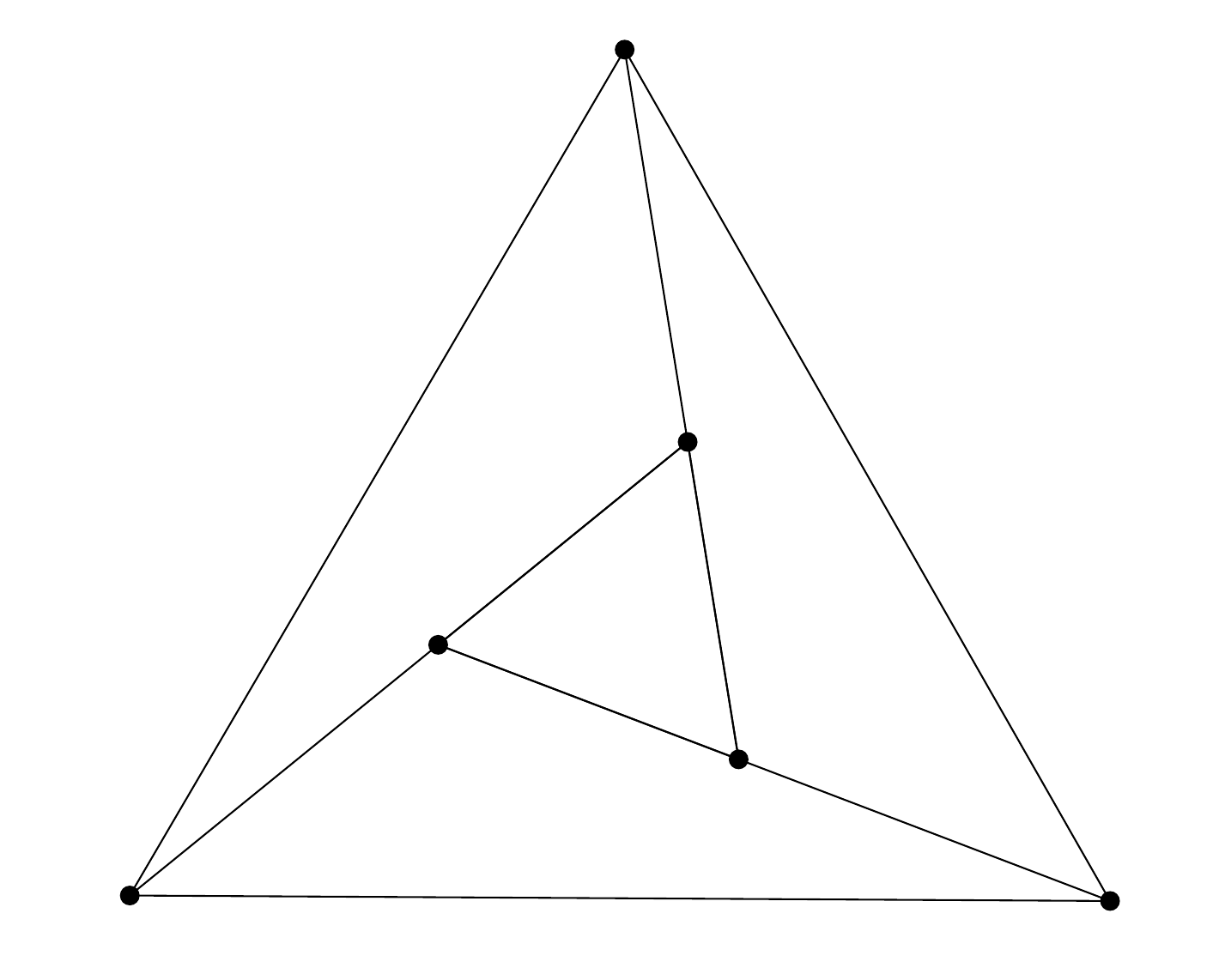}
   \caption{A $3$-fold rotationally symmetric $2$-picture corresponding to the triangular prism. As in Figure~\ref{fig:prismlift}, each region in the figure represents a hyperplane in the combinatorial incidence geometry, and the vertices correspond to points incident to these hyperplanes.}
       \label{fig:C3-sym}
\end{figure}

\begin{ex}
\label{ex:C3-sym_scene}
    Consider the incidence geometry given by a triangular prism and its  realisation as a $2$-picture shown in Figure \ref{fig:C3-sym}. While it has a $3$-fold rotational symmetry in $\mathbb{R}^2$ (denoted $\mathcal{C}_3$ in Schoenflies notation, where $\mathcal{C}_k$ refers to a rotation group of order $k$, and $\mathcal{C}_s$ corresponds to reflection symmetry), its vertical projection to $\mathbb{R}^1$ is not $\mathcal{C}_3$-symmetric, so it is not a lifting of a $\mathcal{C}_3$-symmetric $1$-picture. However, we could consider liftings of the $2$-picture in Figure \ref{fig:C3-sym} to $3$-fold rotationally symmetric scenes in $\mathbb{R}^3$. This $2$-picture is in fact (minimally) flat, and hence admits no non-trivial liftings to $3$-scenes, let alone non-trivial liftings that preserve the $3$-fold rotational symmetry.
\end{ex}

Let $S=(P,L,I)$ be an incidence geometry and let $\Gamma$ be a subgroup of $\Aut(S)$ such that there exists a representation $\tau: \Gamma \to O(d-1)$. A $(d-1)$-picture $(S, x)$ in $\mathbb{R}^{d-1}$ is \emph{$\tau(\Gamma)$-symmetric} if $$\tau(\gamma) x(p) = x(\gamma p)  \qquad \textrm{ for all }  p \in P \textrm{ and } \gamma \in \Gamma.$$
For brevity, we often simply say \emph{$\Gamma$-symmetric} when the representation $\tau$ is clear from the context.

To define symmetric scenes in $\mathbb{R}^d$, we also fix a $1$-dimensional character of $\Gamma$ with values in $\{\pm 1\}$, i.e. a group homomorphism \(\chi:\Gamma\rightarrow\{\pm1\}\). Define the homomorphism $\tau_{\chi}: \Gamma \to O(d)$ by

\[\tau_{\chi}(\gamma) =\begin{pmatrix}
    \tau(\gamma) & 0\\
    0 & \chi(\gamma) \\
\end{pmatrix}.\]

We denote the image of $\tau_{\chi}$ by $\Gamma_{\chi}$; that is, $\Gamma_\chi = \tau_\chi(\Gamma)$, 
so that the action of $\Gamma_\chi$ on $\mathbb{R}^d$ is given by the representation $\tau_\chi$. We will consider scenes that are symmetric with respect to symmetry groups $\Gamma_{\chi}$, where $\Gamma$ has a representation $\tau: \Gamma \to O(d-1)$.

Moreover, define the homomorphism $\hat{\tau}_{\chi}$ by
\[
\hat{\tau}_{\chi}(\gamma) = 
\begin{pmatrix} \tau_\chi(\gamma) & 0 \\ 0 & 1 \end{pmatrix}.\]

\begin{ex}\label{ex:z2ex}
    Let $d=2$, and  let $\Gamma \simeq\mathbb{Z}_2=\{e, \gamma\}$. Let $\tau: \mathbb{Z}_2 \to O(1)$ be defined by $\tau(e)=1$ and $\tau(\gamma)=-1$. 

    First, let $\chi_1$ be the trivial homomorphism, so that $\chi_1(\gamma)=\chi_1(e)=1$. The group $\Gamma_{\chi_1}$ is generated by 
    \[
\tau_{\chi_1}(\gamma) = 
\begin{pmatrix} -1 & 0 \\ 0 & 1 \end{pmatrix}.\]

\noindent
The group $\Gamma_{\chi_1}$ acts on $\mathbb{R}^2$ by a reflection in the $y$-axis.

Next, let $\chi_2$ be given by  $\chi_2(e)=1$ and $\chi_2(\gamma)=-1$. Then $\Gamma_{\chi_2}$ is generated by 
\[
\tau_{\chi_2}(\gamma) = 
\begin{pmatrix} -1 & 0 \\ 0 & -1 \end{pmatrix}.\]

\noindent
The group $\Gamma_{\chi_2}$ acts on $\mathbb{R}^2$ by a half-turn centred at the origin.

 Note that reflection in the $x$-axis cannot arise in this example, unless all points of the $1$-picture are coincident, 
since the first coordinate of each point must be $0$, as $\tau(\gamma)=-1$. 

If instead we assume that $\tau$ is trivial, then choosing the non-trivial character $\chi_2$ would yield reflection symmetry in the $x$-axis for the lifted $2$-scene. In this case, pairs of points of the $1$-picture lying in the same $\mathbb{Z}_2$-orbit  
must coincide in $\mathbb{R}$. 
%\signe{Do we consider that picture to be non-generic wrt $\Gamma$?}\bernd{since $\Gamma$ is trivial, coincident points are not generic w.r.t $\Gamma$. And I think $\Gamma$ is the group we refer to when checking genericity, right? The odd thing here is that if we are non-generic, then the lifting with the larger group does not exist...}\signe{Yeah, it's not completely clear what genericity should mean, I think - combinatorially, the group is still $\mathcal{C}_s$, and wrt that group the points that coincide lie in the same orbits, so it could still be considered $\Gamma$-generic. On the other hand, the representation we're interested in is the trivial representation, and wrt that group action, pictures with coincident points should indeed not be considered generic.}
%\bernd{Gamma-generic def:  coordinates of set of rep vertcies is alg. independent over the rationals. so for x-axis reflection the symmetry just forces the images of vertices to lie on top of rep vertices --  still generic} \signe{All I did here was remove the comment that this is non-generic - do we need to say more?}\bernd{no this is fine I think}
\end{ex}

A $d$-scene over a $\Gamma$-symmetric picture $(S,x)$ is $\Gamma_{\chi}$-symmetric if $$(x(\gamma p), y(\gamma p))^T=\tau_{\chi}(\gamma)(x(p),y(p))^T \qquad \textrm{ for all } p \in P \textrm{  and }\gamma \in \Gamma.$$

If a hyperplane $\ell \in L$ is incident to $k$ points $\{p_1,\ldots,p_k\}$, then in any $d$-dimensional $\Gamma_{\chi}$-symmetric scene, with $d \leq k$, lifting a $\Gamma$-symmetric picture $(S,x)$, the $k$ points $(x(p_i), y(p_i))$ span a unique hyperplane with normal $n(\ell)$ and affine shift $h(\ell)$ in $\mathbb{R}^d$. The hyperplane $\gamma \ell$ is incident to the points $\{\gamma p_1, \ldots ,\gamma p_k\}$, which also span a unique hyperplane. Note that 
\[
\langle\hat{\tau}_{\chi}(\gamma)(x(p_i),y(p_i),1)^T, \hat{\tau}_{\chi}(\gamma) (n(\ell), 1, h(\ell))^T\rangle=\langle(x(p_i),y(p_i),1),(n(\ell),1,h(\ell))\rangle=0
\]
as $\hat{\tau}_{\chi}(\gamma) \in O(d+1)$, so the unique hyperplane spanned by $\{\gamma p_1,\ldots,\gamma p_d\}$ must be the hyperplane determined by $\hat{\tau}_{\chi}(n(\ell),1,h(\ell))^T$. 
%\signe{Here's my suggestion for notation that should make the proof of Theorem \ref{thm:nec_scenes} look a bit nicer.} \bernd{good}
It follows that if that hyperplane corresponding to $\ell \in L$ in a $\Gamma_{\chi}$-symmetric $d$-scene has homogeneous coordinates $[n(\ell):1:h(\ell)]$, then the hyperplane realising $\gamma \ell$ has homogeneous coordinates $[\tau(\gamma)n(\ell): \chi(\gamma): h(\ell)]=[\chi(\gamma) \tau(\gamma) n(\ell) : 1 : \chi(\gamma) h(\ell)]$. As our convention is to set the penultimate homogeneous coordinate of the hyperplanes to $1$, we denote 

\begin{equation}
    (n(\gamma \ell), 1, h(\gamma \ell))^T
=
\chi(\gamma)\hat{\tau}_{\chi}(\gamma)(n(\ell),1,h(\ell))^T. 
\label{eq:lines}
\end{equation}
 for any $\ell \in L$ and $\gamma \in \Gamma$. Note also that  $\hat{\tau}_{\chi}(\gamma)$  acts trivially on the last coordinate of $\mathbb{R}^{d+1}$. Hence 
 $\hat{\tau}_{\chi}(\gamma)(n(\ell),1,h(\ell))^T=(\tau(\gamma) n(\ell), \chi(\gamma), h(\ell))^T$. 

We are now ready to define the orbit lifting matrix of a $\Gamma$-symmetric picture. Let $(S,x)$ be a $\Gamma$-symmetric picture, and fix a character $\chi: \Gamma \to \{\pm 1\}$. The \textit{orbit lifting matrix} of a $\Gamma$-symmetric $(d-1)$-picture (with respect to $\chi$) is the coefficient matrix of the system with one equation of the form 
\begin{equation}
    \langle\hat{\tau_{\chi}}(\gamma)(x(p),y(p),1)^T,(n(\ell),1,h(\ell))\rangle=0.
    \label{eq:sym_scene_2}
\end{equation}
for each orbit $(\ell,\gamma p)$  of incidences. Denote the orbit lifting matrix by $M_{\Gamma_{\chi}}^d(S,x)$.
The row of $M_{\Gamma_{\chi}}^d(S,x)$ corresponding to the incidence $(\ell,\gamma p)$ has the following form:

$$\begin{bNiceMatrix}[first-row,first-col]
         &&&&\ell&&&&&p&&&&\\
		  &0&\ldots&0&(\tau(\gamma)x(p))^T\quad 1&&0&\ldots&0&\chi(\gamma)&0&\ldots&0
        \end{bNiceMatrix}.$$

The kernel of $M_{\Gamma_{\chi}}^d(S,x)$ is in one-to-one correspondence with the $\Gamma_{\chi}$-symmetric scenes lifting $(S,x)$ in the following way:

Take $m \in \mathbb{R}^{d|L|/|\Gamma| + |P|/|\Gamma|}$ and suppose that $m$ consists of one coordinate $m(p)=y(p)$ for each orbit of points, and $d$ coordinates $m(\ell)=(n(\ell),h(\ell))$ for each orbit of hyperplanes. Define $\tilde{m} \in \mathbb{R}^{d|L|+|P|}$ by $\tilde{m}(\gamma p)=\chi(\gamma)y(p)$ for all $p \in P$ and $\gamma \in \Gamma$, and define $\tilde{m}(\gamma\ell)=(\tau(\gamma) n(\ell),h(\ell))$ for all $\ell \in L$ and $\gamma \in \Gamma$.

\begin{thm}
    Let $(S,x)$ be a $(d-1)$-picture,  $\Gamma$ be a subgroup of $\Aut(S)$ and  $\tau: \Gamma \to O(d-1)$ be a representation. Then $m \in \mathbb{R}^{d|L|/|\Gamma|+|P|/|\Gamma|}$ is an element of the kernel of $M^d_{\Gamma_{\chi}}(S,x)$ if and only if $\tilde{m}$ determines a symmetric scene lifting $(S,x)$.

    \label{Thm:scene_orbit_kernel}
\end{thm}

\begin{proof}

Let $(S,x)$ be a $\Gamma$-symmetric $(d-1)$-picture, and consider a 
$\Gamma_\chi$-symmetric $d$-scene lifting $(S,x)$.
The incidence equation in $M_{\Gamma_\chi}^d(S,x)$ corresponding to an incidence $(\ell, \gamma p)$ is
\begin{equation}
\langle (x(\gamma p), y(\gamma p), 1), (n(\ell), 1, h(\ell)) \rangle = 0.
\label{eq:sym_scenes} % (4)
\end{equation}
By the $\Gamma_\chi$-symmetry of the scene, the coordinates of the lifted points satisfy
\[
x(\gamma p) = \tau(\gamma) x(p), \qquad y(\gamma p) = \chi(\gamma) y(p),
\]
so that equation~\eqref{eq:sym_scenes} can be rewritten in terms of the orbit representative $p$ as
\begin{equation}
\langle \hat{\tau}_\chi(\gamma) (x(p), y(p), 1)^T, (n(\ell), 1, h(\ell)) \rangle = 0,
\label{eq:sym_scene_22} % (3)
\end{equation}
which is the same as equation~\eqref{eq:sym_scene_2}.

The equation corresponding to another incidence $(\beta \ell, \beta \gamma p)$ in the same orbit is
\begin{equation}
    \langle(x(\beta \gamma p), y(\beta \gamma p),1),(n(\beta \ell),1,h(\beta \ell)) \rangle=0
    \label{eq:sym_scene_orbit1}
\end{equation}
By the $\Gamma_\chi$-symmetry of the scene, the lifted coordinates satisfy
\[
x(\beta \gamma p) = \tau(\beta \gamma) x(p), \qquad y(\beta \gamma p) = \chi(\beta \gamma) y(p),
\]
and by \eqref{eq:lines}
\[(n(\beta \ell),1,h(\beta \ell))^T=\chi(\beta) \tau_{\chi}(\beta)(n(\ell),1,h(\ell))^T\]
so equation~\eqref{eq:sym_scene_orbit1} can be rewritten in terms of the orbit representative $p$ as
\begin{equation}
\langle \hat{\tau}_\chi(\beta \gamma) (x(p), y(p), 1)^T, \chi(\beta)\hat{\tau}_{\chi}(\beta)(n(\ell), 1, h(\beta \ell))^T \rangle = 0
\label{eq:sym_scene_orbit_2}
\end{equation}
which is equivalent to 
\begin{equation}
    \langle\hat{\tau}_{\chi}(\beta) \hat{\tau}_{\chi}(\gamma)(x(p),y(p),1)^T, \hat{\tau}_{\chi}(\beta)(n(\ell),1,h(\ell))^T\rangle=0
    \label{eq:sym_scene_orbit_2}
\end{equation} 
as $\chi(\beta)$ is a scalar multiple. Since $\hat{\tau}_{\chi}(\beta)$ is an element of $O(d+1)$, equation~\eqref{eq:sym_scene_orbit_2} is equivalent to equation~\eqref{eq:sym_scene_2}. Theorem \ref{Thm:scene_orbit_kernel} now follows from the equivalence of equations \eqref{eq:sym_scene_2} and \eqref{eq:sym_scene_orbit_2}.
\end{proof}

Recall that the kernel of the lifting matrix $M^d(S,x)$ always contains the trivial scenes lifting $(S,x)$, where all elements of $L$ map to the same hyperplane. Analogously, the kernel of $M^d_{\Gamma_{\chi}}(S,x)$ contains the $\Gamma_{\chi}$-symmetric trivial scenes. The $\Gamma_{\chi}$-symmetric trivial scenes are exactly the scenes where all elements of $L$ map to the same hyperplane, except that hyperplane now has to be symmetric under the action of $\Gamma_{\chi}$. Let $H_{\Gamma_{\chi}}$ denote the space of hyperplanes that are invariant under the action of $\Gamma_{\chi}$, i.e. the hyperplanes with normal $n$ and affine shift $h$ such that \[
\langle \hat{\tau}_{\chi}(\gamma) (n,1,h)^T, (x,y,1) \rangle = 0 \quad \text{if and only if} \quad \langle (n,1,h), (x,y,1) \rangle = 0,
\]
or equivalently, in coordinates,
\[
\tau(\gamma) n \cdot x + \chi(\gamma) y + h = 0 \quad \text{if and only if} \quad n \cdot x + y + h = 0.
\]

\begin{prop}
    Let $S=(P,L,I)$ be an incidence geometry and let $\Gamma$ be a subgroup of $\Aut(S)$ with a representation $\tau: \Gamma \to O(d-1)$. Let $\chi: \Gamma \to \{\pm 1\}$ be a one-dimensional character of $\Gamma$.
    
    Then the dimension of $H_{\Gamma_{\chi}}$ is 
    \[
\begin{cases}
 \frac{1}{|\Gamma|}\sum_{\gamma\in\Gamma} 
\mathrm{tr}\!\bigl(\tau(\gamma)\bigr) +1 & \text{if } \chi \text{ is trivial}\\[4pt]
\frac{1}{|\Gamma|}\sum_{\gamma\in\Gamma} 
\chi(\gamma)\,\mathrm{tr}\!\bigl(\tau(\gamma)\bigr) & \text{if } \chi \text{ is non-trivial}.
\end{cases}
\]
\label{prop:dim_trivial_scenes}
\end{prop}

\begin{proof}
    Let \(\tau:\Gamma\rightarrow O(d-1)
\) be the given representation of $\Gamma$, and let \(\chi:\Gamma\rightarrow\{\pm1\}\)
be the $1$-dimensional character of $\Gamma$.
We combine them into the block representation
\[
\tau_{\chi}:\Gamma\longrightarrow O(d),\qquad
\tau_{\chi}(\gamma)
=
\begin{pmatrix}
\tau(\gamma) & 0\\[4pt]
0 & \chi(\gamma)
\end{pmatrix}\in O(d).
\]
We write \(\mathbb R^d=\mathbb{R}^{d-1}\oplus\mathbb R\) and represent a general (non-vertical)  hyperplane by
\[
H=\{(x,y)\in \mathbb{R}^{d-1}\oplus\mathbb R:\; n\cdot x + y = h\},
\]
with normal vector \(v=(n,1)\) and affine shift \(h\).

We observe first that if a hyperplane \(H\) is invariant under \(\tau_{\chi}(\Gamma)\), then its normal \(v\) must satisfy
\(\tau_{\chi}(\gamma)v=\pm v\) for all \(\gamma\in\Gamma\).
From the block structure of \(\tau_{\chi}\) it follows that
\[
\tau(\gamma)\,n=\chi(\gamma)\,n,\qquad\forall \gamma\in\Gamma,
\]
Thus, in the first step, we determine the dimension of the space of such normals, which amounts to finding the multiplicity $\mu_\chi$ of the character \(\chi\) in the representation \(\tau\).
By standard character theory, $\mu_\chi$ is given by
\[
\mu_{\chi} \;=\; \frac{1}{|\Gamma|}\sum_{\gamma\in\Gamma} 
\chi(\gamma)\,\mathrm{tr}\!\bigl(\tau(\gamma)\bigr).
\]
Each copy of \(\chi\) corresponds to one independent normal direction of an invariant hyperplane.

Next we examine the possible affine shifts \(h\) of these hyperplanes.
If \(\chi\) is non-trivial, then there exists some \(\gamma\in\Gamma\) with \(\chi(\gamma)=-1\).
Evaluating the defining equation \(n\cdot x +  y = h\) of \(H\) at the point \((0,h)\)
and applying \(\tau_{\chi}(\gamma)\) shows that this point is sent to \((0,\chi(\gamma)h) = (0,-h)\).
Setwise invariance of \(H\) therefore forces \(h=0\).

If \(\chi\) is trivial, then \(\chi(\gamma)=1\) for all \(\gamma\), and \(h\) remains unconstrained.
Consequently, in this case we obtain one additional parameter corresponding to affine shifts.

In summary, the dimension of the space of unshifted invariant hyperplanes is
\[
\dim = 
\begin{cases}
\mu_{\chi}+1 & \text{if } \chi \text{ is trivial},\\[4pt]
\mu_{\chi} & \text{if } \chi \text{ is non-trivial}.
\end{cases}
\] The proposition follows.
\end{proof}

Note that if $\Gamma_{\chi}'$ is a subgroup of $\Gamma_{\chi}$, we can compute the dimension of the space $H_{\Gamma_{\chi}'}$ of hyperplanes that are invariant under the action of $\Gamma_{\chi}'$ in the same way.

\begin{ex}\label{ex:trivial}
Let $\Gamma = \mathbb{Z}_2 = \{e, \gamma\}$ and consider the non-trivial representation $\tau: \mathbb{Z}_2 \to O(1)$ with $\tau(e)=1$ and $\tau(\gamma)=-1$. (Recall also Example~\ref{ex:z2ex}.) We compute the dimension of invariant hyperplanes for the block representation $\tau_\chi$, where $\chi$ is trivial and non-trivial, respectively. 

\textbf{Case 1:} $\chi$ is  trivial, i.e.\ $\chi(e) = \chi(\gamma) = 1$. The multiplicity of $\chi$ in $\tau$ is
\[
\mu_\chi = \frac{1}{2} \bigl( \chi(e) \, \mathrm{tr}\,\tau(e) + \chi(\gamma) \, \mathrm{tr}\,\tau(\gamma) \bigr) = \frac{1}{2} (1 \cdot 1 + 1 \cdot (-1)) = 0.
\]  
Since $\chi$ is trivial, the affine shift $h$ is unconstrained, giving a total dimension of \( \mu_\chi + 1 = 1\). So the hyperplanes that are invariant with respect to $\tau_{\chi}$ (representing reflection symmetry in the $y$-axis in this case) are the hyperplanes that are parallel to the $x$-axis.

\textbf{Case 2:} $\chi$ is non-trivial, i.e.\ $\chi(e) = 1$, $\chi(\gamma) = -1$. The multiplicity is
\[
\mu_\chi = \frac{1}{2} \bigl( \chi(e) \, \mathrm{tr}\,\tau(e) + \chi(\gamma) \, \mathrm{tr}\,\tau(\gamma) \bigr) = \frac{1}{2} (1 \cdot 1 + (-1) \cdot (-1)) = 1.
\]  
Since $\chi$ is non-trivial, the affine shift $h$ must vanish, so the dimension is
\( \mu_\chi = 1\). The hyperplanes that are invariant with respect to $\tau_{\chi}$ (where $\tau_{\chi}$ acts by a half-turn symmetry in this case) are the hyperplanes that go through the origin. Note that their normals are reversed by the half-turn.
\end{ex}

We say that a $\Gamma$-symmetric $(d-1)$-picture $(S,x)$ is \textit{$\Gamma_{\chi}$-symmetrically flat} if all $\Gamma_{\chi}$-symmetric scenes lifting $(S,x)$ are trivial. The following is an immediate corollary of Theorem \ref{Thm:scene_orbit_kernel} and the discussion above. 

\begin{cor}
        Let $S=(P,L,I)$ be an incidence geometry, and let $(S,x)$ be a $\Gamma$-symmetric $(d-1)$-picture. Then $(S,x)$ is $\Gamma_{\chi}$-symmetrically flat if and only if \[\textup{rank}(M^{d}_{\Gamma_{\chi}}(S, x))=|P|/|\Gamma|+ d|L|/|\Gamma| - \dim H_{\Gamma_\chi}.\]
    \label{cor:flat}
\end{cor}

We end this section with a necessary condition for row independence in the orbit lifting matrix.

\begin{thm}
    Let $S=(P,L,I)$ be an incidence geometry, and let $(S,x)$ be a $\Gamma$-symmetric $(d-1)$-picture. If the rows of $M^d_{\Gamma_{\chi}}(S, x)$ are independent, then
\[
|J| \leq |P(J)| + d|L(J)| 
- \sum_{X \in c(J)} \dim H_{\langle X \rangle_{\psi,p}}
\]
 for all subsets $J \subseteq I_{\Gamma}$, where 
$I_{\Gamma}= I/\Gamma$ denotes the set of $\Gamma$-orbits of incidences, $P(J)$ and $L(J)$ denotes the support of $J$ in the orbit incidence graph, and $c(J)$ denotes the set of connected components of the subgraph of the orbit incidence graph induced by $J$.
    \label{thm:nec_scenes}
\end{thm}
\begin{proof}     
    First, note that if the incidence geometry induced by $J$ is not connected, then 
    \[
|J| = \sum_{X \in c(J)} |X|, \qquad
d|L(J)| = \sum_{X \in c(J)} d|L(X)|, \qquad
|P(J)| = \sum_{X \in c(J)} |P(X)|.
\]    
Hence it suffices to prove the inequality for $J$ inducing a connected incidence geometry.

 Let  $J \subseteq I_\Gamma$, and let $R_J$ be the set of rows of
$M^d_{\Gamma_\chi}(S,x)$ corresponding to $J$. By assumption, $R_J$ is an independent set of rows, so the rank of the submatrix of $M^d_{\Gamma_{\chi}}(S,x)$ given by $R_J$ is $|J|$. We will see that $\dim R_J^\perp \geq \dim H_{\langle J \rangle_{\psi, p}}$, which implies 
\[|J|
\le |P(J)| + d|L(J)| - \dim H_{\langle J \rangle_{\psi,p}}.
\]
  Suppose that the hyperplane with normal $n$ and affine shift $h$ belongs to $H_{\langle J \rangle_{\psi, p}}$.
  Assign $n(\ell)=n$ and $h(\ell)=h$ for all $\ell \in L(J)$.
For each $p \in P(J)$, choose $y(p) \in \mathbb{R}$ such that
\[
\langle (x(p),y(p),1), (n,1,h) \rangle = 0.
\]
Let $(\ell,\gamma p) \in J$. Then
\begin{equation}
\langle \hat{\tau}_\chi(\gamma)(x(p),y(p),1)^T, (n,1,h) \rangle
=
\langle (x(p),y(p),1), \hat{\tau}_\chi(\gamma)^{-1}(n,1,h)^T \rangle.
\label{eq:inv_hyp_vanishes}
\end{equation}
Since $(n,1,h)$ defines a $\langle J \rangle_{\psi,p}$-invariant
hyperplane, we have
\[
\hat{\tau}_\chi(\gamma)^{-1}(n,1,h)^T = (n,1,h)^T,
\]
and therefore the inner product \eqref{eq:inv_hyp_vanishes} vanishes.
Thus every element of $H_{\langle J \rangle_{\psi,p}}$
gives rise to a vector in $R_J^\perp$,
so
\[
\dim R_J^\perp \ge \dim H_{\langle J \rangle_{\psi,p}},
\]
and the theorem follows.
 %   \[
%    <\hat{\tau_{\chi}}(\gamma)(x(p),y(p),1), (n,1,h)>=((x(p),y(p),1), \hat{\tau_{\chi}}^{-1}(\gamma)(n,1,h)>=0
 %   \]
 %   where the final equality holds as $(n,1,h)$ defines a $\langle J \rangle_{\psi, p}$-symmetric hyperplane. Hence $\dim R_J^\perp \geq \dim H_{\langle J \rangle_{\psi, p}}$, so the theorem follows.
\end{proof}

We say that a $\Gamma$-symmetric picture $(S,x)$  is $\Gamma$-generic if the multiset of all coordinates of $\{x(p_1), \dots, x(p_j) \}$ is algebraically independent over $\mathbb{Q}$, where $p_1, \dots, p_j$ form a set of representatives of orbits of points under the action of $\Gamma$. The following is an easy corollary of Theorem \ref{thm:nec_scenes}.

\begin{cor}
    Let $S$ be an incidence geometry, let $\Gamma \leq \mathrm{Aut}(S)$ be a subgroup,
    and let $\chi: \Gamma \to \{\pm 1\}$ be a one-dimensional character. If all $\Gamma$-generic $d-1$-pictures $(S,x)$ are minimally $\Gamma_\chi$-symmetrically flat, then the following hold:
    \begin{enumerate}
    \item $|I|/|\Gamma| = |P|/|\Gamma| + d|L|/|\Gamma| - \dim H_{\Gamma_\chi}$
        %\item $|I_{\Gamma}|=|P_{\Gamma}|+d|L|-\dim(H_{\Gamma})$ and 
        \item  $|J| \leq |P(J)| + d|L(J)| - \Sigma_{X \in c(J)} \dim H_{\langle X \rangle_{\psi, p}}$ for all subsets $J \subseteq I_{\Gamma}$, where $c(J)$ denotes the set of connected components of the graph induced by $J$.
    \end{enumerate}
    \label{cor:sym_flat_nec}
\end{cor}

\begin{ex}[Necessary conditions for $d=2$] \label{ex:scenes_d=2}
Consider dimension $d=2$. Let $\Gamma= \{e, \gamma\}$, and let $\tau(e)=1$ and $\tau(\gamma)=-1$ . 

Suppose that $\chi_1$ is the trivial character, i.e.\ $\chi_1(e) = \chi_1(\gamma) = 1$, so that the combined action $\tau_{\chi_1}$ corresponds to a reflection. As we have seen in Example~\ref{ex:trivial}, $\dim H_{\Gamma_{\chi_1}}=1$ in this case. Similarly, for any unbalanced set of incidence orbits $J$, 
    \(\dim H_{\langle J \rangle_{\psi,p}} = 1\), while for any balanced set of incidence orbits, 
    \(\dim H_{\langle J \rangle_{\psi,p}} = 2\). (For balanced sets, this follows from Theorem~\ref{thm:pic_thm}.)
Hence, in this case, the necessary conditions of Theorem~\ref{thm:nec_scenes}  become

\begin{enumerate}
    \item $|J| \leq |P(J)|+2|L(J)|-1$ for any non-empty unbalanced set of incidence orbits, and 
    \item $|J| \leq |P(J)|+2|L(J)|-2$ for any non-empty balanced set of incidence orbits.
\end{enumerate}

Recall also from Example~\ref{ex:trivial} that for the non-trivial character $\chi_2$ with $\chi_2(\gamma)=-1$, where the combined action $\tau_{\chi_2}$ corresponds to a half-turn,   we also have \(\dim H_{\langle J \rangle_{\psi,p}} = 1\)  for all non-empty unbalanced sets of incidence orbits. For balanced sets, we again have  \(\dim H_{\langle J \rangle_{\psi,p}} = 2\),   by  Theorem~\ref{thm:pic_thm}.  Hence the necessary conditions of Theorem~\ref{thm:nec_scenes} are the same as for the trivial character $\chi_1$. We will return to this point later in Remark~\ref{rmk:pairing} and provide further explanation for why the counting conditions are the same in both cases.
\end{ex}

%\begin{rmk} This pairing of reflection and  half-turn symmetry appearing in Example~\ref{ex:scenes_d=2} is a well-known phenomenon in the rigidity theory of $2$-dimensional bar-joint frameworks. Corollary 5.2 in \cite{cnsw} states that if the action of $\Gamma\cong \mathbb{Z}_2$ is free on the vertex set, then a graph is generically forced $\mathcal{C}_s$-symmetric rigid in the plane if and only if it is generically forced $\mathcal{C}_2$-symmetric rigid in the plane.  Moreover, recall that in the non-symmetric setting, parallel redrawings are known to be equivalent to infinitesimal motions in the plane, and by duality, the same correspondence holds for liftings \bernd{careful here. parallel redrawing gives Maxwell lifting. but lifting to scene?}. We therefore expect this relationship to transfer naturally to the forced symmetric setting.\bernd{hang on. we showed in  \cite{schmil} that fully-sym motion goes to anti-sym parallel redrawing for reflection but to fully-sym parallel redraing in half turn case!  So for the liftings, it reverses again? need to look at this... shall we do an example for parallel drawings case later? }, the phenomenon observed in Example~\ref{ex:scenes_d=2} aligns with these principles and is entirely consistent with established theory. \bernd{add sth on fixed points. though in \cite{cnsw}, there is no result on forced symmetry with fixed points, only for the incidental symmetry isostatic case. but points on mirror go to points at infinity under the transfer. refer back to earlier Remark 3.5 }
%    \end{rmk}

\section{Symmetric hyperplane arrangements and the orbit concurrence geometry matrix}
\label{sec:symmetric_parallel}

\subsection{Symmetric parallel redrawings as duals of symmetric scenes}\label{sec:dualsym}

The duality between scenes and parallel redrawings mentioned in Section \ref{sec:hyperplanes_parallel} (see also \cite{bernd2017sym}) comes from the fact that scenes and parallel redrawings are governed by the same incidence equations
\[
\langle (x(p),y(p),1),(n(\ell),1,h(\ell))\rangle = 0,
\]
but with different variables prescribed. In the problem of finding the scenes that lift a given picture $(S,x)$, the coordinates \(x(p)\in \mathbb{R}^{d-1}\) are fixed and one solves for \(y(p), n(\ell), h(\ell)\). When solving for the parallel redrawings of a hyperplane arrangement, the normals \(n(\ell)\) are fixed and one solves for \(x(p), y(p), h(\ell)\). Thus, the two problems arise from the same bilinear form by fixing complementary sets of variables.

This correspondence can be formalised using projective duality. Liftings are naturally described by the subgroup
\[
G \leq \mathrm{PGL}(d+1,\mathbb{R})
\]
consisting of projective transformations that preserve the projection onto the \((d-1)\)-picture 
while parallel redrawings correspond to a subgroup
\[
D \leq \mathrm{PGL}(d+1,\mathbb{R})
\]
preserving hyperplane normals. Under the duality that sends the point with homogeneous coordinates $[x:y:1]$ to the line with homogeneous coordinates $[x:y:1]$, these two groups are related by the isomorphism
\[
\psi : G \to D, \qquad A \mapsto (A^T)^{-1},
\]
which exchanges points and hyperplanes and preserves incidence. 

Under the duality that sends the point with homogeneous coordinates $[x:y:1]$ to the line with homogeneous coordinates $[x:1:y]$, a scene lifting a picture is sent to a hyperplane arrangement realising the dual incidence geometry with a specified set of normals. In this section, we will use this duality together with the results in Section \ref{sec:sym_scenes} to characterise symmetric parallel redrawings.

Let $S$ be an incidence geometry, let $\Gamma$ be a subgroup of $\Aut (S)$, let $\tau: \Gamma \to O(d-1)$ be a representation of $\Gamma$ and let $\chi$ be a one-dimensional character of $\Gamma$ with values in $\{\pm 1\}$. Suppose that $(S,x)$ is a $\Gamma$-symmetric picture. Recall that a scene that lifts the picture $(S,x)$ is $\Gamma_\chi$-symmetric if and only if
\[(x(\gamma p), y(\gamma p))^T=\tau_\chi(\gamma)(x(p),y(p))^T\]
for all $p \in P$ and $\gamma \in \Gamma$, or, equivalently
\[(x(\gamma p),y(\gamma p),1)^T=\hat{\tau}_\chi(\gamma)(x(p),y(p),1)^T.\]

Dually, we define a hyperplane arrangement $(S,n)$ to be $\Gamma$-symmetric if $n(\gamma \ell)=\tau(\gamma)n(\ell)$ for all $\ell \in L$ and $\gamma \in \Gamma$. We say that a parallel redrawing of $(S,n)$ is $\Gamma_\chi^*$-symmetric if 
\[
(n(\gamma \ell), h(\gamma \ell))^T= \tau_\chi(\gamma)(n(\ell),h(\ell))^T.
\]

We use the convention that the points in $\mathbb{RP}^d$ have homogeneous coordinates $[x:y:1]$, where $x \in \mathbb{R}^{d-1}$ and $y \in \mathbb{R}$. We therefore define $\hat{\tau}_\chi(\gamma)$ in such a way that the action of $\hat{\tau}_\chi(\gamma)$ on $\mathbb{RP}^2$ maintains that the last homogeneous coordinate of each point is $1$. We also use the convention that the lines have homogeneous coordinates $[n:1:h]$, where $n \in \mathbb{R}^{d-1}$ and $h \in \mathbb{R}$. Let

\[\hat{\tau}^*_{\chi}(\gamma) =\begin{pmatrix}
    \tau(\gamma) & 0 & 0\\
    0 & 1 & 0\\
    0 & 0& \chi(\gamma) \\
\end{pmatrix}.\]

Now, the action of $\hat{\tau}^*_{\chi}(\gamma)$ on $\mathbb{RP}^2$ maintains that the penultimate homogeneous coordinate of each line is $1$, and a hyperplane arrangement is $\Gamma_\chi^*$-symmetric if and only if
\[
(n(\gamma \ell),1,h(\gamma \ell))^T=\hat{\tau}^*_{\chi}(\gamma)(n(\ell),1,h(\ell))^T
\]
for all $\ell \in L$ and $\gamma \in \Gamma$.

Recall that the orbit lifting matrix \(M_{\Gamma_\chi}^d(S,x)\) is obtained by imposing the equivariant relation
\[
\langle \hat{\tau}_\chi(\gamma)(x(p),y(p),1)^T,(n(\ell),1,h(\ell))\rangle = 0,
\]
for each orbit of incidences $(\gamma p, \ell) \in I$. Given a $\Gamma$-symmetric set of hyperplane normals $n$, we define the orbit concurrence geometry matrix \(M^{d*}_{\Gamma_\chi^*}(S,n)\) to be the coefficient matrix of the dual system of equations
\[
\langle (x(p),y(p),1),\hat{\tau}^*_{\chi}(\gamma)(n(\ell),1,h(\ell))^T\rangle = 0.
\]

These two systems differ only in which variables are fixed, and the variables acted on by the group element (as the group element acts on the fixed variables). In particular, the convention of directing edges from hyperplane-orbits to point-orbits in the lifting setting, and from point-orbits to hyperplane-orbits in the parallel redrawing setting, reflects precisely this interchange (see Section~\ref{sec:symincgeo}). Consequently, the kernels of \(M_{\Gamma_\chi}^d(S,x)\) and \(M^{d*}_{\Gamma_\chi^*}(S,n)\) -- the spaces of symmetric liftings and symmetric parallel redrawings -- are dual descriptions of the same underlying equivariant incidence system.

Recall that the dual of the incidence geometry $S=(P,L,I)$ is the incidence geometry $S^*=(L,P,I)$. To distinguish between an incidence geometry and its dual, we will denote the line of $S^*$ corresponding to $p \in P$ by $p^*$, and the point of $S^*$ corresponding to $\ell \in L$ by $\ell^*$.

Define the \textit{dual hyperplane arrangement} of a $\Gamma$-symmetric $(d-1)$-picture $(S,x)$ to be the hyperplane arrangement $(S^*, n_x)$ given by by $n_x(p^*)=x(p)$ for all $p \in P$. Similarly, given a $\Gamma_\chi$-symmetric $d$-scene lifting a $\Gamma$-symmetric picture $(S,x)$ we  define the dual parallel redrawing of the hyperplane arrangement $(S,n_x)$ by $h(p^*)=y(p)$, $x(\ell^*)=n(\ell)$ and $y(\ell^*)=h(\ell)$. It is clear that if $(S,x)$ is a $\Gamma$-symmetric picture then $(S^*,n_x)$ is a $\Gamma$-symmetric hyperplane arrangement, and that the dual parallel redrawing of a $\Gamma_\chi$-symmetric scene lifting $(S,x)$ is a $\Gamma_\chi^*$-symmetric parallel redrawing of $(S^*,n_x)$.

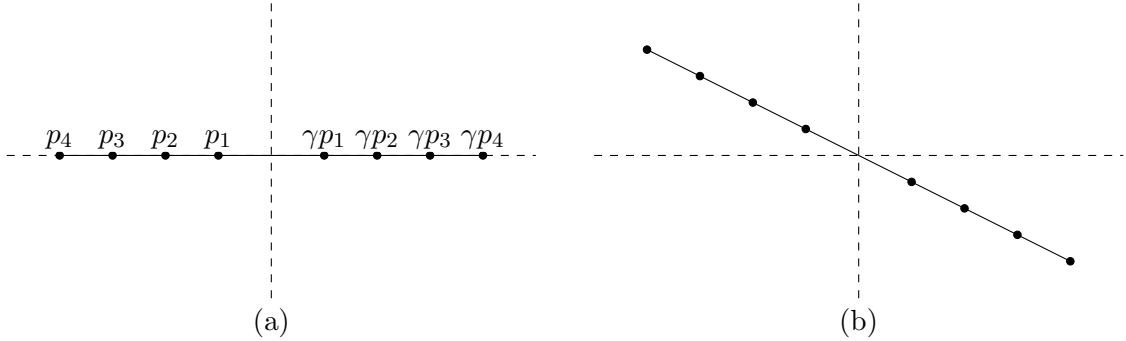
\begin{figure}[htp]
\centering
    \begin{tikzpicture}[scale=0.7]
    \filldraw[black] (-1,0) circle (2pt);
    \filldraw[black] (-2,0) circle (2pt);
    \filldraw[black] (-3,0) circle (2pt);
    \filldraw[black] (-4,0) circle (2pt);
    \filldraw[black] (1,0) circle (2pt);
    \filldraw[black] (2,0) circle (2pt);
    \filldraw[black] (3,0) circle (2pt);
    \filldraw[black] (4,0) circle (2pt);

    \draw[thin, dashed] (-5,0)--(5,0);
    \draw[thin, dashed] (0,-3)--(0,3);
\draw (-4,0)--(4,0);

    \node[fill=white, draw=white, inner sep=1pt, anchor=south] at (-1,0.07) {$p_1$};
    \node[fill=white, draw=white, inner sep=1pt, anchor=south] at (-2,0.07) {$p_2$};
    \node[fill=white, draw=white, inner sep=1pt, anchor=south] at (-3,0.07) {$p_3$};
    \node[fill=white, draw=white, inner sep=1pt, anchor=south] at (-4,0.07) {$p_4$};
    \node[fill=white, draw=white, inner sep=1pt, anchor=south] at (1,0.07) {$\gamma p_1$};
    \node[fill=white, draw=white, inner sep=1pt, anchor=south] at (2,0.07) {$\gamma p_2$};
    \node[fill=white, draw=white, inner sep=1pt, anchor=south] at (3,0.07) {$\gamma p_3$};
    \node[fill=white, draw=white, inner sep=1pt, anchor=south] at (4,0.07) {$\gamma p_4$};

    \node[fill=white, draw=white, inner sep=1pt, anchor=south] at (0,-3.5) {(a)};
\end{tikzpicture}
\hspace{0.5cm}
    \begin{tikzpicture}[scale=0.7]

    \draw[thin, dashed] (-5,0)--(5,0);
    \draw[thin, dashed] (0,-3)--(0,3);

    \filldraw[black] (-1,0.5) circle (2pt);
    \filldraw[black] (-2,1) circle (2pt);
    \filldraw[black] (-3,1.5) circle (2pt);
    \filldraw[black] (-4,2) circle (2pt);
    \filldraw[black] (1,-0.5) circle (2pt);
    \filldraw[black] (2,-1) circle (2pt);
    \filldraw[black] (3,-1.5) circle (2pt);
    \filldraw[black] (4,-2) circle (2pt);

    \draw[thin](-4,2)--(4,-2);

        \node[fill=white, draw=white, inner sep=1pt, anchor=south] at (0,-3.5) {(b)};
\end{tikzpicture}
    \caption{(a) A $\mathbb{Z}_2$-symmetric $1$-picture. (b) A half-turn-symmetric $2$-scene lifting the picture shown in (a).}
    \label{fig:duality_scene}
\end{figure}
\begin{figure}[htp]
    \centering
    \begin{tikzpicture}[scale=0.7]

    \draw[thin, dashed] (-5,0)--(5,0);
    \draw[thin, dashed] (0,-5)--(0,5);

    \draw[thin] (-1.5,4.5)--(1.5,-4.5);
    \draw[thin] (-1.5,-4.5)--(1.5,4.5);
    \draw[thin] (-1,4)--(1,-4);
    \draw[thin] (-1,-4)--(1,4);
    \draw[thin] (-2.5,-5)--(2.5,5);
    \draw[thin] (-2.5,5)--(2.5,-5);
    \draw[thin] (-4,4)--(4,-4);
    \draw[thin] (-4,-4)--(4,4);
    
    \node[fill=white, draw=white, inner sep=1pt, anchor=south] at (2,2) {$p_1^\ast$};
    \node[fill=white, draw=white, inner sep=1pt, anchor=south] at (2,4) {$p_2^\ast$};
    \node[fill=white, draw=white, inner sep=1pt, anchor=south] at (1,3) {$p_3^\ast$};
    \node[fill=white, draw=white, inner sep=1pt, anchor=south] at (1,4) {$p_4^\ast$};
    \node[fill=white, draw=white, inner sep=1pt, anchor=south] at (-2,2) {$\gamma p_1^\ast$};
    \node[fill=white, draw=white, inner sep=1pt, anchor=south] at (-2,4) {$\gamma p_2^\ast$};
    \node[fill=white, draw=white, inner sep=1pt, anchor=south] at (-1,3) {$\gamma p_3^\ast$};
    \node[fill=white, draw=white, inner sep=1pt, anchor=south] at (-1,4) {$\gamma p_4^\ast$};
    \node[fill=white, draw=white, inner sep=1pt, anchor=south] at (0,-5.5) {(a)};
\end{tikzpicture}
\hspace{0.8cm}
\begin{tikzpicture}[scale=0.7]

    \draw[thin, dashed] (-5,0)--(5,0);
    \draw[thin, dashed] (0,-5)--(0,5);

    \draw[thin] (-3,-3.5)--(4,3.5);
    \draw[thin] (-3,3.5)--(4,-3.5);
    \draw[thin] (-2,-5)--(3,5);
    \draw[thin] (-2,5)--(3,-5);
    \draw[thin] (-1,-4.5)--(2,4.5);
    \draw[thin] (-1,4.5)--(2,-4.5);
    \draw[thin] (-0.5,-4)--(1.5,4);
    \draw[thin] (-0.5,4)--(1.5,-4);

    %\node[fill=white, draw=white, inner sep=1pt, anchor=south] at (-1,0.2) {$p_1$};
    
    \node[fill=white, draw=white, inner sep=1pt, anchor=south] at (0,-5.5) {(b)};
\end{tikzpicture}
    \caption{(a) The dual hyperplane arrangement of the symmetric picture in $\mathbb{R}^2$ in Figure \ref{fig:duality_scene}~(a). (b) shows the dual parallel redrawing of the scene in Figure \ref{fig:duality_scene}~(b).}
    \label{fig:duality_parallel}
\end{figure}
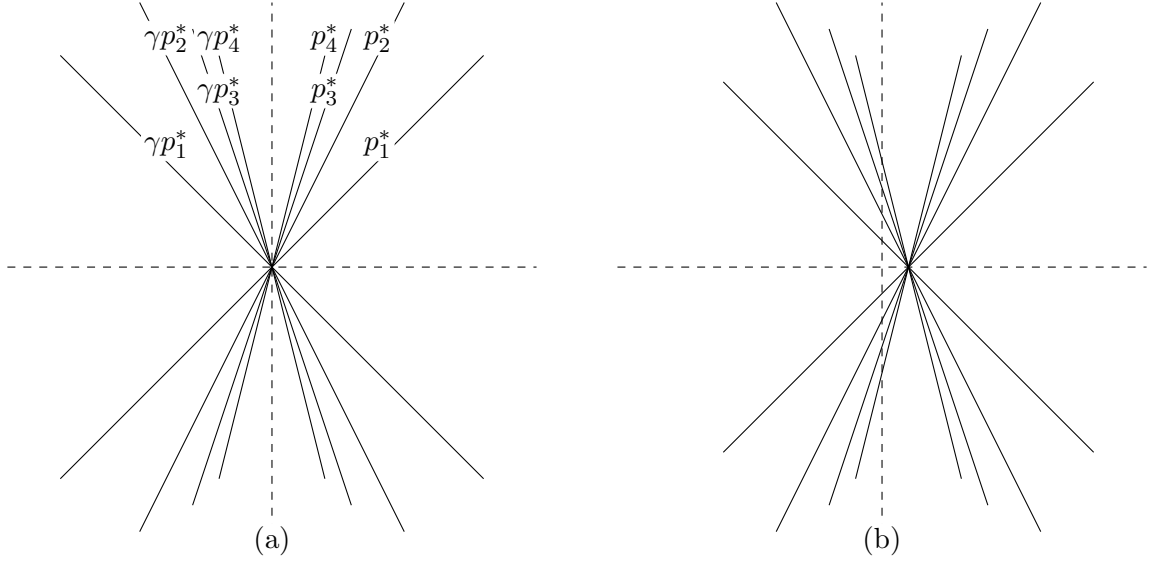

\begin{ex}
    Let $S$ be the incidence geometry with points $P=\{p_1,p_2,p_3,p_4,p_5,p_6,p_7,p_8\}$ and hyperplanes  $L=\{\ell_1,\ell_2,\ell_3,\ell_4,\ell_5,\ell_6\}$, where $\ell_1$ is incident to $p_1$, $p_2$ and $p_7$, $\ell_2$ is incident to $p_1$, $p_4$ and $p_6$, $\ell_3$ is incident to $p_2$, $p_3$ and $p_4$, $\ell_4$ is incident to $p_3$, $p_5$ and $p_6$, $\ell_5$ is incident to $p_2$, $p_5$ and $p_8$, and $\ell_6$ is incident to $p_6$, $p_7$ and $p_8$. 

    There is an automorphism $\gamma$ of order $2$ of $S$ that maps 
    \[
    p_1 \mapsto p_5, \qquad
    p_2 \mapsto p_6, \qquad
    p_3 \mapsto p_7, \qquad
    p_4 \mapsto p_8
    \]
    and
    \[
    \ell_1 \mapsto \ell_4, \qquad
    \ell_2 \mapsto \ell_5 \qquad 
    \ell_3 \mapsto \ell_6.
    \]

    Let $\Gamma = \langle \gamma \rangle$, and $\tau(\gamma)=-1$. See Figure~\ref{fig:duality_scene} (a) for an example of a $\Gamma$-symmetric $1$-picture $(S,x)$. The coordinates are given by $x(p_1)=[-1:1]$, $x(p_2)=[-2:1]$, $x(p_2)=[-3:1]$ and $x(p_4)=[-4:1]$, where we consider $(S,x)$ as a picture in $\mathbb{RP}^1$. Note that the points of the picture in Figure~\ref{fig:duality_scene} (a) are in the plane, and have homogeneous coordinates given by $x(p_1)=[-1:0:1]$, $x(p_2)=[-2:0:1]$, $x(p_2)=[-3:0:1]$ and $x(p_4)=[-4:0:1]$, and $\gamma p_i$ has homogeneous coordinates $x(\gamma p_i)=[-x(p_i):0:1]$, for $1 \leq i \leq 4$. The picture in Figure \ref{fig:duality_scene} (a) is in fact embedded as the trivial scene lifting the picture $(S,x)$, where the lines are all assigned the hyperplane $y=0$.
    
    Let $\chi$ be the non-trivial character of $\Gamma$, so that $\chi(\gamma)=-1$. The $\Gamma_\chi$-symmetric scenes lifting $(S,x)$ are invariant with respect to the action of 
    \[\hat{\tau}_\chi(\gamma)=\begin{pmatrix}
    -1 & 0 & 0\\
    0 & -1 & 0\\
    0 & 0& 1 \\
\end{pmatrix}\]
    on $\mathbb{RP}^2$, so $\Gamma_\chi$-symmetric scenes lifting $(S,x)$ are $\mathcal{C}_2$-symmetric. The orbit lifting matrix $M_{\Gamma_\chi}^2(S,x)$ is
    \[\begin{bNiceMatrix}[first-row,first-col]
         &p_1&p_2&p_3&p_4&\ell_1&\ell_2&\ell_3\\
		  &1&0&0&0&-1 \quad 1&0 \quad 0& 0 \quad 0\\
          &1&0&0&0&0 \quad 0&-1 \quad 1& 0 \quad 0\\
          &0&1&0&0&-2 \quad 1&0 \quad 0& 0 \quad 0\\
          &0&-1&0&0&0 \quad 0&2 \quad 1& 0 \quad 0\\
          &0&1&0&0&0 \quad 0&0 \quad 0& -2 \quad 1\\
          &0&0&-1&0&3 \quad 1&0 \quad 0& 0 \quad 0\\
          &0&0&1&0&0 \quad 0&0 \quad 0& -3 \quad 1\\
          &0&0&0&1&0 \quad 0&-4 \quad 1& 0 \quad 0\\
          &0&0&0&1&0 \quad 0&0 \quad 0& -4 \quad 1\\
        \end{bNiceMatrix}.\]

        It can be easily verified that the kernel of $M_{\Gamma_\chi}^2(S,x)$ is $1$-dimensional. In fact, the kernel of $M_{\Gamma_\chi}^2(S,x)$ consists of vectors of the form
        \[v_\lambda=\begin{bNiceMatrix}[first-row]
        p_1 & p_2 & p_3 & p_4 & \ell_1 & \ell_1 & \ell_2 & \ell_2 & \ell_3 & \ell_3 & \ell_4 & \ell_4\\
         \lambda & 2\lambda & 3\lambda & 4\lambda & \lambda & 0 & \lambda & 0 & \lambda & 0 & \lambda & 0\\
        \end{bNiceMatrix}^T\]
        where $\lambda \in \mathbb{R}$. The elements of the kernel correspond to the $1$-dimensional space of trivial scenes given by mapping all elements of $\ell$ to the line $\lambda x + y=0$ is, and mapping the point $p_i$ to $(-i, \lambda i)$, for $i \in \{1,2,3,4\}$, and $(i-4, -\lambda (i-4))$ for $i \in \{5,6,7,8\}$. Figure \ref{fig:duality_scene} (b) shows an example of a scene lifting the picture in Figure \ref{fig:duality_scene} (a).

        Now consider the dual hyperplane arrangement $(S^*, x^*)$, pictured in Figure \ref{fig:duality_parallel} (a), given by $n_x(p_1^*)=-1$, $n_x(p_2^*)=-2$, $n_x(p_3^*)=-3$, $n_x(p_4^*)=-4$, $n_x(p_5^*)=n_x(\gamma p_1^*)=1$, $n_x(p_6^*)=n_x(\gamma p_2^*)=2$, $n_x(p_7^*)=n_x(\gamma p_3^*)=3$ and $n_x(p_8^*)=n_x(\gamma p_4^*)=4$. In Figure \ref{fig:duality_parallel}, the hyperplanes are drawn in $\mathbb{R}^2$ as the degenerate parallel redrawing where $p_i^*$ is assigned the line with homogeneous coordinates $[-i:1:0]$, for $i \in \{1,2,3,4\}$, and $[i-4:1:0]$ for $i \in \{5, 6,7,8\}$, the dual hyperplane arrangement of the picture in Figure \ref{fig:duality_scene} (a). The points $\ell_i^*$ are all assigned the point $[0:0:1]$, the dual of the hyperplanes assigned to the lines $\ell_i$ in the trivial scene also in Figure \ref{fig:duality_scene}.
        
        The $\Gamma_\chi^*$-symmetric parallel redrawings of $(S^*,x^*)$ are invariant with respect to the action of 
        \[\hat{\tau}^*_\chi(\gamma)=\begin{pmatrix}
    -1 & 0 & 0\\
    0 & 1 & 0\\
    0 & 0& -1 \\
    \end{pmatrix}\]
    on $\mathbb{RP}^2$. Note that the action of $\hat{\tau}^*_\chi(\gamma)$ on $\mathbb{RP}^2$ is the same as the action of 
    \[\begin{pmatrix}
    1 & 0 & 0\\
    0 & -1 & 0\\
    0 & 0& 1 \\
    \end{pmatrix}\]
    on $\mathbb{RP}^2$ which sends the point $[x:y:1]$ to the point $[x:-y:1]$. Hence $\Gamma_\chi$-symmetric parallel redrawings of $(S^*,x^*)$ are symmetric with respect to reflection in the $x$-axis, rather than rotationally symmetric. The orbit concurrence geometry matrix \(M^{2*}_{\Gamma_\chi}(S^*,n_x)\) is
    \[\begin{bNiceMatrix}[first-row,first-col]
         &p_1^*&p_2^*&p_3^*&p_4^*&\ell_1^*&\ell_2^*&\ell_3^*\\
		  &1&0&0&0&-1 \quad 1&0 \quad 0& 0 \quad 0\\
          &1&0&0&0&0 \quad 0&-1 \quad 1& 0 \quad 0\\
          &0&1&0&0&-2 \quad 1&0 \quad 0& 0 \quad 0\\
          &0&-1&0&0&0 \quad 0&2 \quad 1& 0 \quad 0\\
          &0&1&0&0&0 \quad 0&0 \quad 0& -2 \quad 1\\
          &0&0&-1&0&3 \quad 1&0 \quad 0& 0 \quad 0\\
          &0&0&1&0&0 \quad 0&0 \quad 0& -3 \quad 1\\
          &0&0&0&1&0 \quad 0&-4 \quad 1& 0 \quad 0\\
          &0&0&0&1&0 \quad 0&0 \quad 0& -4 \quad 1\\
        \end{bNiceMatrix}.\]

        The kernel of $M_\Gamma^{2*}(S^*,n_x)$ again consists of the vectors $v_\lambda$. In this setting, $v_\lambda$ corresponds to translation along the $x$-axis. Figure \ref{fig:duality_parallel} (b) shows the dual parallel redrawing of the scene in Figure \ref{fig:duality_scene} (b).
        \label{ex:dual_symmetry}
    \end{ex}
    
The fact that the scene and the dual hyperplane arrangement in Example \ref{ex:dual_symmetry} have different symmetries is a consequence of the fact that a $\Gamma_\chi$-symmetric scene is invariant under the action of $\hat{\tau}_\chi(\gamma)$ on the projective plane, while the dual symmetric hyperplane arrangement is invariant under the action of $\hat{\tau}_\chi^*(\gamma)$ on the projective plane. The possible actions of the matrices $\hat{\tau}_\chi(\gamma)$ and $\hat{\tau}_\chi^*(\gamma)$ on $\mathbb{RP}^2$ are summed up in Table~\ref{tab:groupaction}.

\begin{table}[htp]
\begin{center}
\begin{tabular}{|c|c|c|}
\hline
 & Action of $\hat{\tau}_\chi(\gamma)$ &Action of $\hat{\tau}_\chi^*(\gamma)$ \\
 \hline
$\tau(\gamma)=-1$, $\chi(\gamma)=1$ & Reflection in the $y$-axis & Reflection in the $y$-axis\\
\hline
$\tau(\gamma)=-1$, $\chi(\gamma)=-1$ & Order $2$ rotation & Reflection in the $x$-axis\\
\hline 
$\tau(\gamma)=1$, $\chi(\gamma)=-1$ & Reflection in the $y$-axis & Order $2$ rotation\\
\hline
\end{tabular}
\caption{The actions of the matrices $\hat{\tau}_\chi(\gamma)$ and $\hat{\tau}_\chi^*(\gamma)$ on $\mathbb{RP}^2$.}
  \label{tab:groupaction}
\end{center}
\end{table}

In general, $\hat{\tau}_\chi(\gamma)$ and $\hat{\tau}_\chi^*(\gamma)$ act the same on $\mathbb{RP}^d$ if $\chi$ is trivial, and differently if $\chi$ is non-trivial.
  
\begin{prop}[Duality between symmetric liftings and symmetric parallel redrawings]\label{prop:dual}
Let \(S=(P,L,I)\) be a \(\Gamma\)-symmetric incidence geometry, let \(\tau:\Gamma\to O(d-1)\) be a representation, and let \(\chi:\Gamma\to\{\pm1\}\) be a character. Let \((S,x)\) be a \(\Gamma\)-symmetric \((d-1)\)-picture, and let \(S^*=(L,P,I)\) denote the dual incidence geometry.

Then, under projective duality, there is a linear isomorphism between
\begin{itemize}
    \item the space of \(\Gamma_\chi\)-symmetric liftings of \((S,x)\), and
    \item the space of \(\Gamma_\chi^*\)-symmetric parallel redrawings of the dual hyperplane arrangement $(S^*,n_x)$.
\end{itemize}

In particular, the kernels of the orbit lifting matrix \(M^d_{\Gamma_\chi}(S,x)\) and the corresponding orbit concurrence geometry matrix of $(S^*,n_x)$ are isomorphic.
\end{prop}

\begin{proof}
A \(\Gamma_\chi\)-symmetric scene that lifts the picture \((S,x)\) is given by orbit variables
\[
\bigl(y(p)\bigr)_{p\in P/\Gamma} \qquad \bigl(n(\ell),h(\ell)\bigr)_{\ell\in L/\Gamma}
\]
satisfying
\[
\langle \hat{\tau}_\chi(\gamma)(x(p),y(p),1)^T,(n(\ell),1,h(\ell))\rangle =0
\]
for every orbit of incidences \((\ell,\gamma p)\). By Theorem~\ref{Thm:scene_orbit_kernel}, the $\Gamma_\chi$-symmmetric scenes that lift $(S,x)$ are precisely the elements of
\(\ker M^d_{\Gamma_\chi}(S,x)\).

Now let $(S^*,n_x)$ be the dual hyperplane arrangement on \(S^*\) defined by assigning to each
\(p\in P/\Gamma\) the hyperplane normal
\[
n(p^*):=x(p).
\]
The $\Gamma_\chi^*$-symmetric parallel redrawings correspond to solutions to the system of equations
\[
\langle(x(\ell^*),y(\ell^*),1),\hat{\tau}^*(\gamma)(n( p^*),1,h(p^*))^T\rangle=0
\]
for all incidences $(\gamma p,\ell) \in I$, where $x(\ell^*)$, $y(\ell^*)$ and $h(\ell^*)$ are variables. Given a $\Gamma_\chi$-symmetric scene lifting $(S,x)$, the dual parallel redrawing is a solution to the system of equations above by definition. Similarly, given a solution to the system of equations above, the $d$-scene given by $y(p)=h(p^*)$, $n(\ell)=x(\ell^*)$ and $h(\ell)=y(\ell^*)$ is $\Gamma_\chi$-symmetric by definition.

Equivalently, the kernels of the orbit lifting matrix $M_{\Gamma_\chi}^d(S,x)$ and the orbit concurrence geometry matrix $M_{\Gamma_\chi}^{d*}(S^*,n_x)$ are isomorphic.
\end{proof}

Let us now determine the space of trivial $\Gamma_\chi^*$-symmetric parallel redrawings of a $\Gamma$-symmetric hyperplane arrangement. Such a redrawing consists of translations in the point coordinates together with a global dilation, subject to the symmetry constraints on the geometric data
\[
n(\gamma \ell)=\tau(\gamma)\,n(\ell), \qquad h(\gamma \ell)=\chi(\gamma)\,h(\ell),
\]
while the vertex coordinate functions $x(p)\in \mathbb{R}^{d-1}$ and $y(p)\in \mathbb{R}$ are not themselves equipped with a group action; instead, symmetry is enforced through compatibility with the incidence equations.
A trivial parallel redrawing is generated by a translation $(a,b)\in \mathbb{R}^{d-1}\times \mathbb{R}$ together with a global dilation $\lambda\in \mathbb{R}$, acting as
\[
x'(p)=x(p)+a,\qquad y'(p)=y(p)+b,\qquad h'(\ell)=h(\ell)-n(\ell)\cdot a - b,
\]
and
\[
x'(p)=\lambda x(p),\qquad y'(p)=\lambda y(p),\qquad h'(\ell)=\lambda h(\ell).
\]
Imposing symmetry on the translated data requires that the transformed affine shifts $h'$ still satisfy
\[
h'(\gamma \ell) = \chi(\gamma)\, h'(\ell).
\]
Substituting the translation formulas and using $n(\gamma \ell) = \tau(\gamma)\, n(\ell)$ shows that this holds if and only if the translation vector $a$ satisfies
\[
\tau(\gamma)\, a = \chi(\gamma)\, a \quad \text{for all } \gamma \in \Gamma.
\]
So the dimension of the  space of symmetric $x$-translations is given by
\[
\frac{1}{|\Gamma|}\sum_{\gamma\in\Gamma}\chi(\gamma)\mathrm{tr}(\tau(\gamma)).
\]
For the vertical translation parameter $b$, we consider the induced transformation
\[
x'(p)=x(p), \qquad y'(p)=y(p)+b, \qquad h'(\ell)=h(\ell)-b.
\]
To preserve symmetry of the affine shift data, we again require
\[
h'(\gamma \ell)=\chi(\gamma)\,h'(\ell).
\]
Substituting the definition of $h'$ and using $h(\gamma \ell)=\chi(\gamma)h(\ell)$ yields
\[
\chi(\gamma)h(\ell)-b=\chi(\gamma)h(\ell)-\chi(\gamma)b,
\]
which implies
\[
b=\chi(\gamma)b \quad \text{for all } \gamma\in\Gamma.
\]
Hence $b$ is unconstrained if $\chi$ is trivial, and $b=0$ if $\chi$ is non-trivial.
Finally, dilation is always $\Gamma_\chi^*$-symmetric since it acts by uniform scalar multiplication on all variables and commutes with the group action.
It follows that the dimension of the space of trivial $\Gamma_\chi^*$-symmetric parallel redrawings is
\[
\frac{1}{|\Gamma|}\sum_{\gamma\in\Gamma}\chi(\gamma)\mathrm{tr}(\tau(\gamma))
\;+\;
\begin{cases}
2 & \text{if } \chi \text{ is trivial},\\[4pt]
1 & \text{if } \chi \text{ is non-trivial}.
\end{cases}
\]

\begin{rmk} \label{rmk:pairing}
We are now ready to explain the pairing of reflection and half-turn symmetry appearing in Example~\ref{ex:scenes_d=2}, which is related to a well-known feature in the (infinitesimal) rigidity theory of $2$-dimensional bar-joint frameworks with reflection (or half-turn) symmetry: infinitesimal motions decompose into symmetry types corresponding to irreducible representations of the symmetry group. In the case of reflection symmetry, they decompose into forced-symmetric ones corresponding to $\chi_1$ (velocity vectors are mirror  images of each other) and anti-symmetric ones corresponding to $\chi_2$ (velocity vectors are mirror images of each other but with a sign change) \cite{bernd2017sym,schmil}.
We revisit Example~\ref{ex:scenes_d=2} in the case when the incidence geometry is a simple graph and the $1$-picture is a ($1$-dimensional) bar-joint framework.

\textbf{(i) $\chi_1$ (trivial character, reflection symmetry):} Here we consider forced-symmetric liftings, where $\chi_1$ is trivial. Geometrically, this means  that we lift the $1$-dimensional framework by vertically extending the ``mirror'' (point of inversion in dimension $1$)  to the next higher dimenion, i.e., the $y$-axis in the plane. As shown in Proposition~\ref{prop:dual},  symmetric liftings of the $1$-picture ($1$-dimensional framework) with this reflection symmetry correspond to  symmetric parallel redrawings of the dual hyperplane arrangement ($2$-dimensional framework) with respect to the  same reflection symmetry  (recall Table~\ref{tab:groupaction}). Note that the space of trivial symmetric liftings is $1$-dimensional, whereas the dimension of trivial symmetric parallel redrawings of the dual structure is $2$-dimensional.  As shown in \cite{schmil}, symmetric parallel redrawings of a framework with respect to reflection symmetry in the plane correspond to  infinitesimal motions of the opposite symmetry type, i.e. anti-symmetric motions.  In particular, the space of trivial anti-symmetric infinitesimal motions with respect to reflection symmetry is also $2$-dimensional.

\textbf{(ii) $\chi_2$ (non-trivial character, half-turn):} Here we again consider forced-symmetric liftings of the $1$-picture, but now $\chi_2$ is non-trivial, so the  two vertices in each orbit of the lifting are forced to have opposite signs in both coordinates. In other words, the resulting $2$-scene must have half-turn symmetry. Symmetric liftings in this setting correspond to  reflection-symmetric parallel redrawings of the dual hyperplane arrangement, where the reflection is in the $x$-axis (recall Table~\ref{tab:groupaction} and Example~\ref{ex:dual_symmetry}). While the space of trivial symmetric liftings with half-turn symmetry is $1$-dimensional, the dimension of trivial symmetric parallel redrawings of the dual structure with reflection symmetry is again $2$-dimensional. Again, as shown in \cite{schmil}, symmetric parallel redrawings of a framework with respect to reflection symmetry correspond to anti-symmetric infinitesimal motions. 

Thus, in both cases, the problem reduces to considering \emph{anti-symmetric infinitesimal rigidity} of $2$-dimensional reflection-symmetric frameworks.  
%Since forced-symmetric rigidity is known to transfer between mirror and half-turn symmetries (Corollary 5.2 in~\cite{cnsw}), the same holds for anti-symmetric rigidity. Therefore, the observed correspondences between liftings, parallel redrawings, and infinitesimal motions are entirely consistent.
%Note that for larger symmetry groups in higher dimensions, different characters typically lead to different dimensions of the corresponding spaces $H_{\Gamma_\chi}$ and hence to different sparsity conditions.
\end{rmk}

\begin{rmk}
    Scene analysis and finding non-trivial parallel redrawings can both be seen as  special cases of projective rigidity of configurations of points and lines, obtained by constraining either the normal of the lines to be fixed or by constraining the first coordinate of each point. 
    Consequently both the orbit lifting matrix and the orbit concurrency geometry matrix can  be obtained from the orbit projective rigidity matrix explained in the recent article \cite{proj_paper1} on projective rigidity fixing certain variables.
\end{rmk}

\subsection{Parallel redrawings with symmetry groups in $O(d)$}\label{sec:pardrgen}

By the duality described in the previous section, we can obtain results for symmetric parallel redrawings that are dual to the results in Section \ref{sec:sym_scenes}. When considering liftings of $(d-1)$-pictures to $d$-scenes it is natural to consider symmetry groups in $O(d-1)$, as those are the possible symmetry groups of the $(d-1)$-pictures. As described in Section \ref{sec:sym_scenes}, we can then obtain the possible symmetry groups of scenes lifting a symmetric picture, by extending the symmetry group of the picture to a symmetry group in $O(d)$ by a one-dimensional character with values in $\{\pm 1\}$. Parallel redrawings with symmetry groups that are given by a symmetry group in $O(d-1)$ extended by a character can be considered as the dual to symmetric scenes.

Not all discrete subgroups of $O(d)$ can be obtained from a subgroup of $O(d-1)$ in the way described in Section \ref{sec:sym_scenes}. For example, the order $3$ rotational symmetry in Example \ref{ex:C3-sym_scene} cannot be obtained in that way. However, we could consider parallel redrawings which preserve order $3$ rotational symmetry. See Figure~\ref{fig:C3-sympar} for an example.

\begin{figure}[htp]
    \centering
    \includegraphics[width=0.49\linewidth]{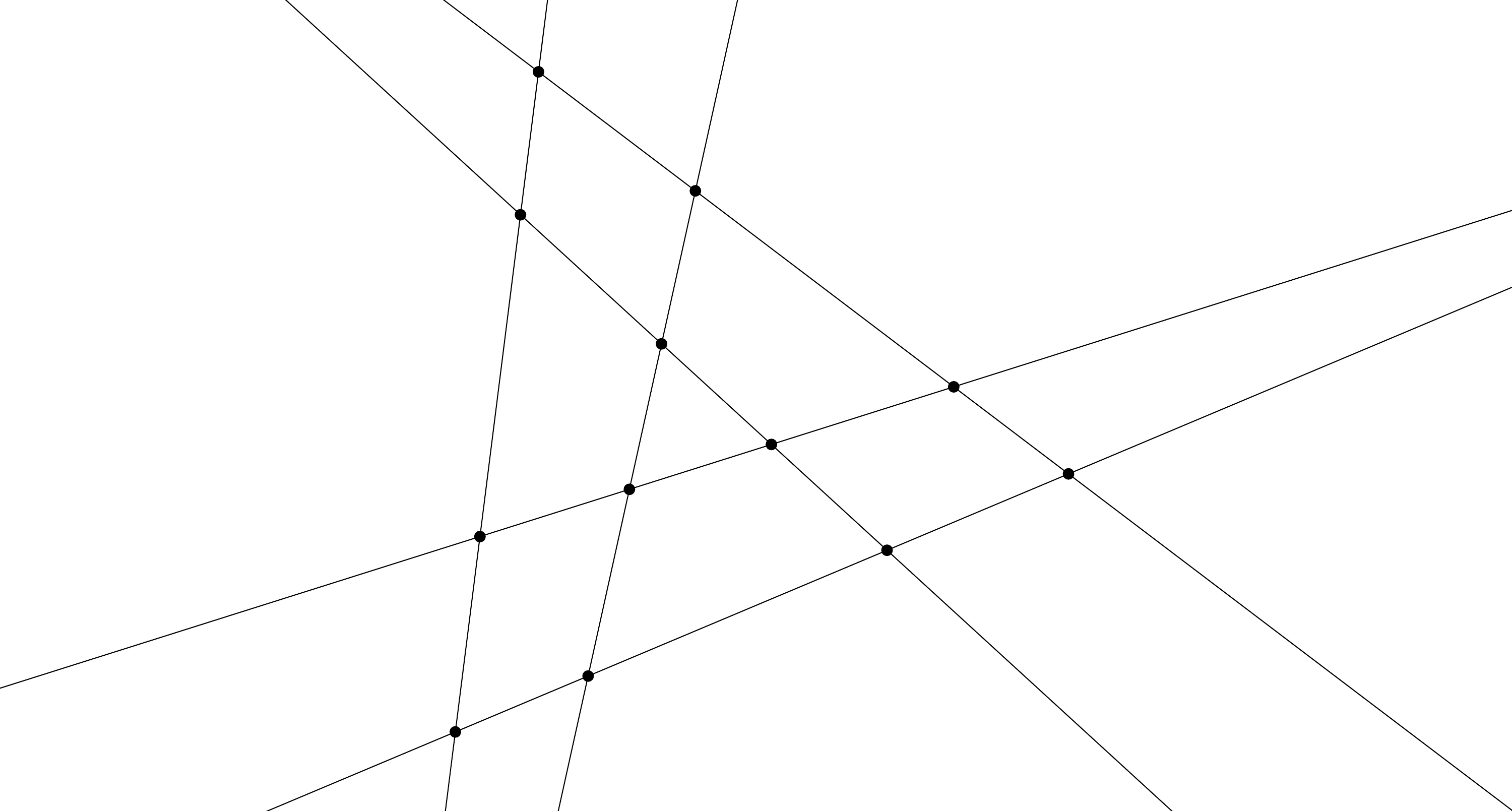}
    \includegraphics[width=0.49\linewidth]{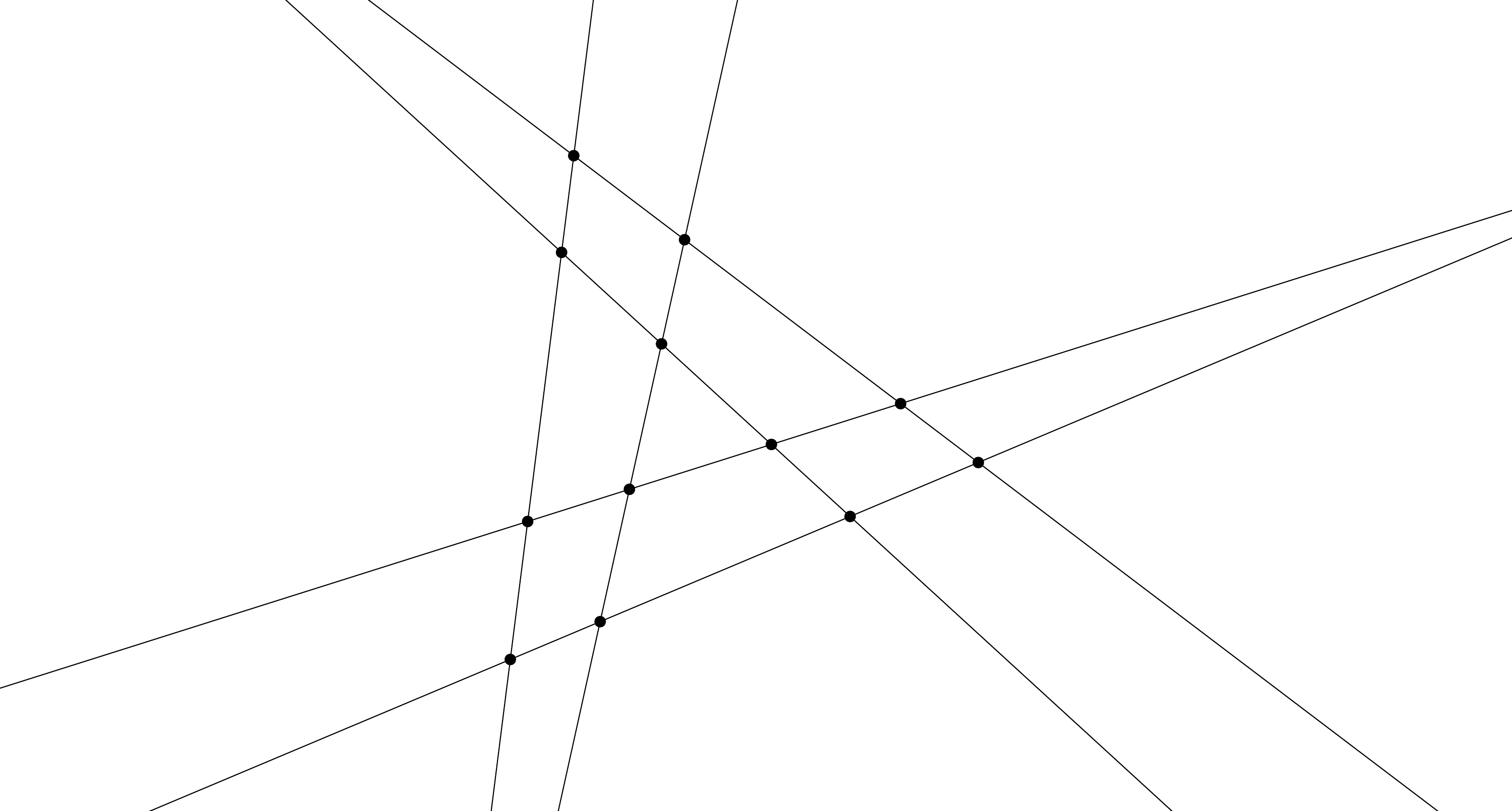}
   \caption{A $3$-fold rotationally symmetric hyperplane arrangement which has a symmetry-preserving parallel redrawing.}
    \label{fig:C3-sympar}
\end{figure}

In this section we extend our results for symmetric parallel redrawings to include all symmetry groups in $O(d)$. In order to be able to formally consider all symmetry groups in $O(d)$, we need to slighty alter our set-up. In the setting introduced in Section \ref{sec:hyperplanes_parallel}, we take $n(\ell) \in \mathbb{R}^{d-1}$, and we consider an equation
\[
    \langle(x(p),y(p),1), (n(\ell),1,h(\ell))\rangle=
    0
\]
for all incidences $(p, \ell) \in I$. One can instead consider $n(\ell) \in \mathbb{R}^d$, and an equation 

\begin{equation}
    \langle(x(p),y(p),1), (n(\ell),h(\ell))\rangle=
    0
    \label{eq:parallel_d}
\end{equation}
for each incidence $(p,\ell) \in I$. Under the mild assumption that the $d$:th coordinate of $n(\ell)$ is non-zero for all $\ell \in L$, considering equations of the form \eqref{eq:parallel_d} is equivalent to the original set-up introduced in Section \ref{sec:hyperplanes_parallel}, since the representation of the hyperplane in homogeneous coordinates is invariant under scaling. 

We define a hyperplane arrangement $(S,n)$ of an incidence geometry $S=(P,L,I)$ to be an assignment of a hyperplane normal $n(\ell) \in \mathbb{R}^d$ to each $\ell \in L$. Let $\Gamma$ be a subgroup of $\Aut (S)$ and let $\tau: \Gamma \to O(d)$ be a representation of $\Gamma$.

A hyperplane arrangement is \emph{$\Gamma$-symmetric} if $n(\gamma\ell)=\tau(\gamma)n(\ell)$ for all $\gamma \in \Gamma$ and $\ell \in L$. Define $\hat{\tau}(\gamma) \in O(d+1)$ by 
\[\hat{\tau}(\gamma) =\begin{pmatrix}
    \tau(\gamma) & 0\\
    0 & 1 \\
\end{pmatrix}.\]

\begin{rmk} 

In this section, we consider parallel redrawings where the hyperplanes and point-sets are invariant under the action of matrices of the form $\hat{\tau}(\gamma)$,
%\[\hat{\tau}(\gamma) =\begin{pmatrix}
%    \tau(\gamma) & 0\\
%    0 & 1 \\
%\end{pmatrix}\]
where $\tau(\gamma) \in O(d)$. The matrices $\hat{\tau}_\chi^*(\gamma)$ considered in Section \ref{sec:dualsym} have the form
\[
\begin{pmatrix}
\tau(\gamma) & 0 & 0\\
0 & 1 & 0\\
0 & 0 &\chi(\gamma)
\end{pmatrix}
\]
where $\tau(\gamma) \in O(d-1)$, so they do not take quite the same form as the matrices that we consider in this section. However, the matrix 
\[
\begin{pmatrix}
\chi(\gamma) \tau(\gamma) & 0 & 0\\
0 & \chi(\gamma) & 0\\
0 & 0 & 1
\end{pmatrix}
\]
is of the form that we consider in this section, and it acts in the same way as $\hat{\tau}_\chi^*(\gamma)$ on $d$-dimensional projective space, since the matrices only differ by a scalar multiple.
\end{rmk}

A parallel redrawing of $(S,n)$ is $\Gamma$-symmetric if
\begin{equation}
    \langle(x(p),y(p),1),\hat{\tau}(\gamma)(n(\ell),h(\ell))^T\rangle=0
    \label{eq:sym_parallel2}
\end{equation}
for all orbits of incidences $(p, \gamma \ell)$, where $x(p)$, $y(p)$ and $h(\ell)$ are variables.

Going back to Example \ref{ex:C3-sym_scene}, we notice that the point coordinates of the configuration are also $\mathcal{C}_3$-symmetric. In general, if a hyperplane arrangement $(S,n)$ is $\Gamma_{\chi}$-symmetric, then the point coordinates of a $\Gamma_{\chi}$-symmetric parallel redrawing of $(S,n)$ are also $\Gamma_{\chi}$-symmetric. Specifically, the point coordinates $(x(p),y(p))$ satisfy
\begin{equation}
    (x(\gamma p), y(\gamma p),1)^T=\hat{\tau}(\gamma)(x(p),y(p),1)^T
    \label{eq:sym_points}
\end{equation}
for all $\gamma \in \Gamma$ and $p \in P$.

Define the orbit concurrence geometry matrix $M_{\Gamma}^{d*}(S,n)$ to be the coefficient matrix of the system of equations of the form $\eqref{eq:sym_parallel2}$. The symmetric parallel redrawings of $(S,n)$ are elements of the kernel of the orbit concurrence geometry matrix. To see this, take $m \in \mathbb{R}^{|L|/|\Gamma| + d|P|/|\Gamma|}$ and suppose that $m$ consists of one coordinate $m(\ell)=h(\ell)$ for each orbit of lines, and $d$ coordinates $m(p)=(x(p),y(p))$ for each orbit of points. Define $\tilde{m} \in \mathbb{R}^{|L|+d|P|}$ by $\tilde{m}(\gamma \ell)=h(\ell)$ for all $\ell \in L$ and $\gamma \in \Gamma$, and define $\tilde{m}(\gamma p)=\tau(\gamma)(x(p), y(p))$ for all $p \in P$ and $\gamma \in \Gamma$.

\begin{thm}
    Let $(S,n)$ be a symmetric hyperplane arrangement. Then $m \in \mathbb{R}^{|L|/|\Gamma|+d|P|/|\Gamma|}$ is an element of the kernel of $M^{d*}_{\Gamma}(S,n)$ if and only if $\tilde{m}$ determines a symmetric parallel redrawing of $(S,n)$.
    \label{thm:sym_kernel_O(d)}
\end{thm}

\begin{proof}
    The proof is analogous to the proof of Theorem \ref{Thm:scene_orbit_kernel}. 
\end{proof}
\begin{comment}
        
    Let $(S,n)$ be a $\Gamma$-symmetric hyperplane arrangement, and consider a $\Gamma$-symmetric parallel redrawing of $(S,n)$. Then the equation in $M^{d*}(S,n)$ corresponding to the incidence $(p, \gamma \ell)$ is

\begin{equation}
    \langle(x(p),y(p),1), (n(\gamma \ell),h(\gamma \ell))\rangle=
    0.
    \label{eq:sym_parallel}
\end{equation}

    As $(S,x)$ is $\Gamma$-symmetric, equation~\eqref{eq:sym_parallel} is equivalent to equation \eqref{eq:sym_parallel2}.

The equation corresponding to another incidence $(\beta p, \beta \gamma \ell)$ in the same orbit is

\begin{equation}
    <(x(\beta p), y(\beta p),1),(n(\beta \gamma \ell),h(\beta \gamma \ell)) >=0
    \label{eq:sym_scene_orbit}
\end{equation}

which is equivalent to 

\begin{equation}
    <\hat{\tau}(\beta) (x(p),y(p),1), \hat{\tau}(\beta) \hat{\tau}(\gamma) (n(\ell),h(\ell))>=0.
    \label{eq:sym_parallel_orbit_2}
\end{equation}

As $\hat{\tau}(\beta)$ is an element of $O(d+1)$, equation~\eqref{eq:sym_parallel_orbit_2} is equivalent to equation~\eqref{eq:sym_parallel2}. The conclusion follows.
\end{comment}

The kernel of $M^{d*}_{\Gamma}(S,n)$ contains at least the translations that preserve $\Gamma$-symmetry. Let $T_{\Gamma}$ denote the space of such translations. The dimension of $T_{\Gamma}$ is the multiplicity of the trivial representation in $\tau$, which can be computed as follows:

\[
\dim(T_{\Gamma})=\frac{1}{|\Gamma|}(\Sigma_{\gamma \in \Gamma} \text{tr}(\tau(\gamma))).
\]

We now obtain the following necessary conditions for row independence in $M^{d*}_{\Gamma}(S,n)$:

\begin{thm}
    Let $S=(P,L,I)$ be an incidence geometry, and let $(S,n)$ be a $\Gamma$-symmetric hyperplane arrangement. If the rows of $M^{d*}_{\Gamma}(S, n)$ are independent, then the following holds:
    $$
    |J| \leq |L(J)| + d|P(J)| - \Sigma_{X \in c(J)} \dim T_{\langle X \rangle_{\psi, p}}$$
    for all subsets $J \subseteq I_{\Gamma}$, where $c(J)$ denotes the set of connected components of the subgraph generated by $J$.
    \label{thm:nec_sym_parallel}
\end{thm}

\begin{proof}
    The proof is analogous to the proof of Theorem \ref{thm:nec_scenes}.
\end{proof}

If the $\Gamma$-symmetric hyperplane arrangement $(S,n)$ has only trivial parallel redrawings, then the $\Gamma$-symmetric translations are the only elements of the kernel of $M^{d*}_{\Gamma}(S,n)$. On the other hand, if the hyperplane arrangement has some non-trivial parallel redrawing, then dilation is also an element of the kernel of $M^{d*}_{\Gamma}(S,n)$. We say that a $\Gamma$-symmetric hyperplane arrangement $(S,n)$ is \textit{$\Gamma$-symmetrically robust} if the only $\Gamma$-symmetric parallel redrawings of $(S,n)$ are translations and dilations.

A hyperplane arrangement is \textit{$\Gamma$-generic} if the multiset $\{n(\ell_1),\ldots,n(\ell_j)\}$ is algebraically independent over $\mathbb{Q}$, where $\{\ell_1,\ldots,\ell_j\}$ is a set of representatives of the orbits of lines under the action of $\Gamma$. The following corollary is immediate from the above discussion and Theorem \ref{thm:nec_sym_parallel}. 

\begin{cor}
    Let $S=(P,L,I)$ be an incidence geometry, and let $\Gamma$ be a group acting on $S$. Then if all $(d-1)$-dimensional $\Gamma$-generic hyperplane arrangements $(S,n)$ are $\Gamma$-symmetrically robust the following holds:
    \begin{enumerate}
        \item $|I_{\Gamma}|=|L_{\Gamma}|+d|P_{\Gamma}|-(1+\frac{1}{|\Gamma|}(\Sigma_{\gamma \in \Gamma} \text{tr}(\tau(\gamma))))$ and 
        \item  $|J| \leq |L_{\Gamma}(J)| + d|P_{\Gamma}(J)| - (1+\Sigma_{X \in c(J)} \frac{1}{|\langle X \rangle_{\psi, p}|}(\Sigma_{\gamma \in \langle X \rangle_{\psi, p}} \text{tr}(\tau(\gamma)))$ for all subsets $J \subseteq I_{\Gamma}$ with $|J| \geq 2$, where $c(J)$ denotes the set of connected components of the graph generated by $J$.
    \end{enumerate}
    \label{cor:sym_robust_nec}
\end{cor}

Finally, we end the section with a further comment on why we do not consider arbitrary symmetry groups in $O(d)$ in the dual situation for scenes. In the original setting, we have $x(p) \in \mathbb{R}^{d-1}$, and we consider an equation
\[
    \langle(x(p),y(p),1), (n(\ell),1,h(\ell))\rangle=
    0
\]
for all incidences $(p, \ell) \in I$. Dually to the set-up for parallel redrawings, one can instead consider $x(p) \in \mathbb{R}^d$, and an equation 

\begin{equation}
    \langle(x_1(p),\ldots,x_{d-1}(p),y(p),x_d(p)), (n(\ell),1,h(\ell))\rangle=
    0
    \label{eq:pic_d}
\end{equation}
for each incidence $(p,\ell) \in I$. Again, under the mild assumption that $x_d(p) \neq 0$, this is equivalent to the original set-up in Section \ref{sec:pic_scenes}. The assumption that $x_d \neq 0$ means that the point $x(p)$ is not a point at infinity of $\mathbb{RP}^{d-1}$. We will see that in this case, the symmetries of the pictures do not transfer neatly to symmetries of scenes in Euclidean space.

\begin{ex}
Consider $1$-pictures of the graph $K_3$ with vertex set $\{p_1,p_2,p_3\}$ and edges 
$\{p_1,p_2\}$, $\{p_1,p_3\}$, $\{p_2,p_3\}$. Note that the group $\mathcal{C}_3$ is a subgroup of the automorphism group of $K_3$. Let $\Gamma=\mathcal{C}_3$, and let $\gamma=(p_1p_2p_3)$ be the generator of $\mathcal{C}_3$. 
Define $$\tau(\gamma)=\begin{pmatrix}
    \cos{\frac{\pi}{3}} & \sin{\frac{\pi}{3}} \\
    -\sin{\frac{\pi}{3}} & \cos{\frac{\pi}{3}} \\
\end{pmatrix}.$$ Let us define a $\Gamma$-symmetric $1$-picture of $K_3$.
%
%Let $x(p_1)=[1:1] \in \mathbb{RP}^1$. Then $x(p_2)=x(\gamma p_1)=\tau(\gamma)x(p)=[\frac{1+\sqrt{3}}{2}:\frac{1-\sqrt{3}}{2}]$. Similarly, the point $x(p_3)=x(\gamma^2 p_1)=\tau(\gamma)^2x(p_1)=[\frac{-1+\sqrt{3}}{2}:\frac{-1-\sqrt{3}}{2}]$. The corresponding affine points in $\mathbb{R}^1$ are the points $x(p_1)=1$, $x(p_2)=\frac{1+\sqrt{3}}{1-\sqrt{3}}$ and $x(p_3)=\frac{-1+\sqrt{3}}{\frac{-1-\sqrt{3}}{2}}$. 
Let $x(p_1)=[1:1] \in \mathbb{RP}^1$. Then
\[
x(p_2)=x(\gamma p_1)=\tau(\gamma)x(p_1)=\left[\frac{1+\sqrt{3}}{2}:\frac{1-\sqrt{3}}{2}\right].
\]
Rescaling homogeneous coordinates, this is equivalent to
\[
x(p_2)=[1:\tfrac{1-\sqrt{3}}{1+\sqrt{3}}]=[1:-2+\sqrt{3}].
\]
Similarly,
\[
x(p_3)=x(\gamma^2 p_1)=\tau(\gamma)^2 x(p_1)=\left[\frac{-1+\sqrt{3}}{2}:\frac{-1-\sqrt{3}}{2}\right],
\]
which simplifies to
\[
x(p_3)=[1:\tfrac{-1-\sqrt{3}}{-1+\sqrt{3}}]=[1:-2-\sqrt{3}].
\]
Thus, the corresponding affine points in $\mathbb{R}^1$ (using representatives of the form $[1:t]$) are
\[
x(p_1)=1,\quad x(p_2)=-2+\sqrt{3},\quad x(p_3)=-2-\sqrt{3}.
\]

% Note that $\tau(\gamma)$ acts on $\mathbb{R}^2$ by a rotation of order $3$ centred at the origin, whereas the picture is defined  in $\mathbb{RP}^{1}$. So while symmetry of the homogeneous coordinates under the action of the matrix $\tau(\gamma)$ does impose conditions on the points of the picture, in the sense that the symmetry of homogeneous coordinates means that the position of all points in an orbit can be determined from one representative of the orbit, symmetry of the homogeneous coordinates under the action of $\tau(\gamma)$ does not impose a Euclidean symmetry of the $1$-picture.

Note that $\tau(\gamma)$ acts on $\mathbb{R}^2$ as a rotation of order $3$ centred at the origin, whereas the picture is defined in $\mathbb{RP}^1$. Thus, symmetry of the homogeneous coordinates under the action of $\tau(\gamma)$ imposes conditions on the points of the picture, in the sense that all points in an orbit are determined by a single representative. However, this symmetry does not induce a Euclidean symmetry of the resulting $1$-picture in $\mathbb{R}^1$.

%Furthermore, let $\chi: \Gamma \to \{\pm 1\}$ be the trivial character. The matrix $\tau_{\chi}$ acting on the scenes over $(G,x)$ would then be \[\tau_{\chi}(\gamma)=\begin{pmatrix}
%    \cos{\frac{\pi}{3}} & 0 & \sin{\frac{\pi}{3}} \\
%    0 & 1 & 0\\
%    -\sin{\frac{\pi}{3}} & 0 & \cos{\frac{\pi}{3}} \\
%\end{pmatrix}.\] 

%Note that if we consider scenes over the picture $x$ that are symmetric with respect to the action of $\tau_\chi$ on $\mathbb{PR}^2$, then the normals of the lines would not have a $\mathcal{C}_3$-symmetry.
\end{ex}

\begin{comment}
\begin{figure}[h]
    \centering
    \includegraphics[width=0.5\linewidth]{Figures/ref_sym.pdf}
    \caption{A configuration with a parallel redrawing that preserves reflectional symmetry}
    \label{fig:ref_parallel}
\end{figure}
\end{comment}

\section{Relation to symmetric parallel redrawings of graphs}
\label{sec:robustness}

In this section, we relate the results of this paper to existing results about parallel redrawings of graphs. Let $G=(V,E)$ be a graph. A \textit{$d$-dimensional framework} $(G,\rho)$ of a graph is an assignment $\rho: V \to \mathbb{R}^d$. A framework is $\textit{generic}$ if the coordinates of the points are algebraically independent over $\mathbb{Q}$.

A framework $(G, \rho')$ is a \textit{parallel redrawing} of a framework $(G, \rho)$ if $\rho(v_i)-\rho(v_j)$ is parallel to $\rho'(v_i)-\rho'(v_j)$ for all edges $(v_i,v_j) \in E$. In other words, a framework $(G,\rho')$ is a parallel redrawing of a framework $(G,\rho)$ if 
\begin{equation}
    \langle(\rho(v_i)-\rho(v_j))^{\perp}, \rho'(v_i)-\rho'(v_j)\rangle=0
\label{eq:graph_parallel}
\end{equation}

\noindent
for all edges $(v_i,v_j) \in E$. Finding all parallel redrawings amounts to solving a system of $|E|$ equations of the form \eqref{eq:graph_parallel}, where $\rho'(v)$ consists of $d$ variables for all $v \in V$. If $d \geq 3$, then finding parallel redrawings of a  framework $(G, \rho)$, i.e. finding frameworks $(G, \rho')$ where the edges are parallel to the edges in the original framework, is not the same as finding parallel redrawings of a hyperplane arrangement. So in this section, we restrict to dimension $d=2$.

Let the \textit{parallel design matrix} $D(G,\rho)$ be the coefficient matrix of the system of $|E|$ equations of the form \eqref{eq:graph_parallel}. Suppose that $e=(v_1,v_2)$ is an edge of $G$, and let $\rho(v_1)=(x_1, y_1)$ and $\rho(v_2)=(x_2,y_2)$. Assuming that all points in the framework are distinct, the line segment between $\rho(v_1)$ and $\rho(v_2)$ is a segment of a line $f_ex+y+h$, where $f_e=-\frac{y_2-y_1}{x_2-x_1}$. Let $n_{\rho}=\{f_e \mid e \in E \}$. Now, considering parallel redrawings of $(G,n_{\rho})$ as elements of the kernel of the matrix $M^*(G, n_\rho)$ and as elements of the kernel of $D(G, \rho)$ is equivalent. In fact, assuming that all points in the framework are distinct, the matrix $M^{2*}(G,n_\rho)$ can be transformed to
\[\begin{pmatrix} I_{|E|} & 0 \\ 0 & D(G,\rho) \end{pmatrix}\]
by row-reductions [\cite{WW}, Example 4.3.1]. In this section, we will see that a similar relationship holds between the corresponding orbit matrices, with some restrictions on the symmetry groups.

First, we define the orbit parallel design matrix. Let $\Gamma$ be a subgroup of $\Aut (G)$ and let $\tau: \Gamma \to O(2)$ be a representation. We will again assume that the action of $\Gamma$ on $G$ is free. A framework $(G, \rho)$ is $\Gamma$-symmetric if $\rho(\gamma v)=\tau(\gamma) v$ for all $\gamma \in \Gamma$ and $v \in V$. A parallel redrawing $(G, \rho')$ of a $\Gamma$-symmetric framework $(G, \rho)$ is $\Gamma$-symmetric if $(G, \rho')$ is $\Gamma$-symmetric. 

Suppose that $(G, \rho)$ is $\Gamma$-symmetric, and let $(v_i, \gamma v_j) \in E$. The equation in $D(G,\rho)$ corresponding to $(v_i, \gamma v_j)$ is
\begin{equation}
    \langle(\rho(v_i)-\rho(\gamma v_j))^{\perp},(\rho'(v_i)-\rho'(\gamma v_j)\rangle=0.
    \label{eq:graph_parallel_1}
\end{equation}

As $(G, \rho)$ and $(G, \rho')$ are both $\Gamma$-symmetric and $\tau(\gamma)$ is orthogonal, equation \eqref{eq:graph_parallel_1} can be rewritten as
\begin{equation}
    \langle(\rho(v_i)-\tau(\gamma)\rho(v_j))^{\perp}, \rho'(v_i) \rangle - \langle(\tau(\gamma)^{-1}\rho(v_i)-\rho(v_j))^{\perp},\rho'(v_j)\rangle=0
    \label{eq:graph_parallel_sym}
\end{equation}
for each orbit of edges $\{v_1, \gamma v_2\}$. Let the \textit{orbit parallel design matrix} $D_{\Gamma, \tau}(G, \rho)$ be the coefficient matrix of the system of equations of the form \eqref{eq:graph_parallel_sym}. The following characterisation of independence in $D_{\Gamma, \tau}(G,\rho)$ is due to Tanigawa, where $\Gamma$-generic is defined as in Section~\ref{sec:sym_scenes}.

\begin{thm}[\cite{tanigawamatroids}, Corollary 6.4]
    Let $G$ be a graph and let $\Gamma$ be a subgroup of $\Aut(G)$ with a representation $\tau: \Aut(G) \to O(2)$. Then $\Gamma$-generic frameworks $(G, \rho)$ are minimally $\Gamma$-symmetrically robust if and only if
    \begin{enumerate}
        \item $|E|/|\Gamma|=2|V|/|\Gamma|-1-\frac{1}{|\Gamma|}\Sigma_{\gamma \in \Gamma} \textup{tr}(\tau(\gamma))$, and 
        \item $|E'| \leq 2V_{\Gamma}(E')-2C(E')-1+\Sigma_{X \in C(E')}\dim\{ \textup{im}(I_2-\tau(\gamma)) \mid \gamma \in \langle X \rangle \}$ for all edges $E'$ in the quotient gain graph $(G/\Gamma, \psi)$.
    \end{enumerate}
    \label{thm:tanigawa_graphs}
\end{thm}

Now, let $G=(V,E)$ be a graph and let $\Gamma \subseteq \Aut(G)$ be a subgroup. Further, let $\tau: \Gamma \to O(2)$ be a representation of $\Gamma$. Suppose that $(G, \rho)$ is a $\Gamma$-symmetric framework. Given an edge $(v_i,v_j) \in E$, the line segment between $\rho(v_i)$ and $\rho(v_j)$ is a segment of a line $f_ex+y+h$, where $f_e=-\frac{y_j-y_i}{x_j-x_i}$ and $h$ is some real number. Equivalently, up to a scalar multiple, the line segment between $\rho(v_i)$ and $\rho(v_j)$ can be written as
\[-(y_j-y_i)x+(x_j-x_i)y+h'=0,\]
where $h'$ is some real number. Define $n(e)=(-(y_i-y_j),x_i-x_j)=(\rho(v_i)-\rho(v_j))^{\perp}$ for all $e \in E$. Note that if $(G, \rho)$ is $\Gamma$-symmetric, then so is the hyperplane arrangement $(G,n)$. We can therefore consider its $\Gamma$-symmetric parallel redrawings. The following relationship holds between $M_{\Gamma}^{2*}(G, n)$ and $D_{\Gamma}(G,\rho)$.

\begin{prop}

Let $G=(V,E)$ be a graph, and let $\Gamma \subseteq \Aut(G)$ be a subgroup with a representation $\tau: \Gamma \to O(2)$. Let $(G,\rho)$ be a $\Gamma$-symmetric framework. Then, up to row operations, we have
    \[M_{\Gamma}(G, n)=\begin{pmatrix} I_{|E|} & 0 \\ 0 & D_{\Gamma}(G,\rho) \end{pmatrix}.\]
\end{prop}

\begin{proof}
Each orbit of edges $e=(v_i, \gamma v_j)$ of $G$ contributes two orbits of incidences, $(v_i, e)$ and $(v_j, \gamma^{-1} e)$. The rows of $M^{2*}_{\Gamma_{\chi}}(G,n)$ corresponding to these two incidences are

 \setcounter{MaxMatrixCols}{20}
  \[\begin{bNiceMatrix}[first-row,first-col]
         &&&&v_i & &v_j&&&&&e&&\\
		  &0&\ldots&0&0&0&\tau(\gamma)^{-1}n(e)&0&0&\ldots&0&1&0&\ldots&0 \\
        &0&\ldots&0&n(e)&0&0&0&0&\ldots&0&1&0&\ldots&0 
        \end{bNiceMatrix}.\]

\noindent
%Removing the first row from the second gives the following two rows
After replacing the second row by the difference of the second and first rows, and reordering the rows, we obtain
 \setcounter{MaxMatrixCols}{20}
  \[\begin{bNiceMatrix}[first-row,first-col]
         &&&&v_i &  &v_j&&&&&e&&\\
		  &0&\ldots&0&n(e)&0&0&0&0&\ldots&0&1&0&\ldots&0 \\
        &0&\ldots&0&n(e)&0&-\tau(\gamma)^{-1}n(e)&0&0&\ldots&0&0&0&\ldots&0 
        \end{bNiceMatrix}.\] 

\noindent
Now, $n(e)=(\rho(v_i)-\tau(\gamma)(\rho(v_j))^{\perp}$, and $\tau^{-1}(\gamma)n(e)=(\tau(\gamma)^{-1}\rho(v_i)-\rho(v_j))^{\perp}$, so the row 
 \setcounter{MaxMatrixCols}{20}
  \[\begin{bNiceMatrix}[first-row,first-col]
         &&&&v_i &  &v_j&&&&&e&&\\
        &0&\ldots&0&n(e)&0&-\tau(\gamma)^{-1}n(e)&0&0&\ldots&0&0&0&\ldots&0 
        \end{bNiceMatrix}\] 
can be recognised as the row corresponding to the edge orbit $(v_i, \gamma v_j)$ in the matrix $D_{\Gamma}(G, \rho)$.

Hence, up to a permutation of the rows, the orbit concurrence geometry matrix reduces to 
\[M_{\tau}(G, n)=\begin{pmatrix} I_{|E|} & 0 \\ 0 & D_{\Gamma}(G,\rho) \end{pmatrix}.\]
\end{proof}

\section{An alternative approach and conjectures on sufficiency of the necessary conditions}
\label{sec:conjectures}
The main  question related to symmetric liftings and parallel redrawings that remains open  is whether the necessary conditions established in this paper are sufficient under suitable genericity assumptions. We make the conjecture that, under our assumption that the group action is free,
%the conditions set in this paper, 
they are always sufficient (assuming $\Gamma$-genericity). 

\begin{conj}
The necessary conditions of Corollary \ref{cor:sym_flat_nec} and Corollary \ref{cor:sym_robust_nec} are sufficient for $\Gamma$-generic pictures and hyperplane arrangements. 
    \label{conj:sufficient}
\end{conj}

Tanigawa proved Conjecture \ref{conj:sufficient} for graphs in the plane (recall Theorem \ref{thm:tanigawa_graphs}). Tanigawa's proof of Theorem \ref{thm:tanigawa_graphs} and its extension to higher dimensions utilises \textit{Dowling geometries}, a type of matroid on a gain graph which incorporates the geometric group action \cite{tan}. However, the correspondence between symmetric parallel redrawings of hyperplane arrangements and symmetric parallel designs of graphs only holds in the plane. 

For the remainder of this section, we suggest a different way to set up the problem of finding the symmetric scenes lifting a symmetric picture, which may be useful in understanding sufficiency of the necessary condition, following \cite{WW89}. 

Whiteley's approach to proving the picture theorem (Theorem \ref{thm:pic_thm}) involves first proving the theorem for $(d-1)$-pictures of $(d+1)$-uniform hypergraphs, where each hyperedge is incident to $d+1$ points. Dually, we can set up the same system of equations for parallel redrawings of $(d+1)$-regular hypergraphs, where every point belongs to $d+1$ hyperedges, with minor modifications if the symmetry group $\Gamma$ of the parallel redrawing is not represented by $\tau_{\chi}$ for some one-dimensional character $\chi: \Gamma \to \{\pm 1\}$.
%cannot be obtained from a subgroup of $O(d-1)$ by extending it by a character. 

Now, the points of a hyperedge $e=\{p_1,\ldots,p_{d+1}\}$ are coplanar if and only if

\begin{equation}
    \det \begin{pmatrix}
x(p_1) & \cdots & \cdots & x(p_{d+1}) \\
y(p_1) & \cdots & \cdots & y(p_{d+1}) \\
1 & 1 & 1 & 1
\end{pmatrix}=0.
\label{eq:scenes_determinant_form}
\end{equation}
Finding the $d$-scenes lifting $(d-1)$-pictures of a $(d+1)$-uniform hypergraph now amounts to solving a system of equations in the variables $y(p_i)$. A similar system of equations can be set up for \emph{symmetric} pictures. Specifically, the representative hyperedge of an hyperedge orbit,  $e=\{\gamma_1 p_1,\ldots,\gamma_{d+1} p_{d+1}\}$, can be represented by the equation

\begin{equation}
\det \begin{pmatrix}
\tau(\gamma_1) x(p_1) & \cdots & \cdots &\tau(\gamma_{d+1}) x(p_{d+1}) \\
\chi(\gamma_1)y(p_1) & \cdots & \cdots & \chi(\gamma_{d+1}) y(p_{d+1}) \\
1 & 1 & 1 & 1
\end{pmatrix}=0.
\label{eq:scenes_determinant_orbit}
\end{equation}

Using a similar proof as in Theorem~\ref{Thm:scene_orbit_kernel}, it is easy to see that the equation corresponding to another hyperedge $\beta e= \{\beta \gamma_1p_1,\ldots,\beta\gamma_{d+1}p_{d+1}\}$ in the same orbit as $e$ is equivalent to equation \eqref{eq:scenes_determinant_orbit}.

This gives an alternative system of equations for $\Gamma$-symmetric $(d+1)$-uniform hypergraphs that can be investigated using methods inspired by those in \cite{WW89}.  %The alternative system of equations also invites the following conjecture, which is weaker than Conjecture~\ref{conj:sufficient}:
Here we do not require the group to be orthogonal, and hence this approach applies to any subgroup of the general linear group. Moreover, this holds even if a hyperedge has a non-trivial stabiliser under the group action.

Note that the determinant formulation exhibits exceptional cases in which an equation of the form \eqref{eq:scenes_determinant_orbit}  becomes identically satisfied, as we will see below. Such degeneracies arise in the hypergraph formulation and present a potential obstruction to adapting Whiteley's hypergraph approach to the symmetric setting. This motivates the following weaker conjecture.

\begin{conj}
    The necessary conditions of Corollary \ref{cor:sym_flat_nec} and Corollary \ref{cor:sym_robust_nec} are sufficent for $\Gamma$-generic pictures and hyperplane arrangements, provided that equation \eqref{eq:scenes_determinant_orbit} is not satisfied for all choices of $y(p_i)$ for some orbit of edges $e=\{\gamma_1p_1,\ldots,\gamma_{d+1}p_{d+1}\}$.
    \label{conj:sufficient_alternative}
\end{conj}

%We note that it can happen that equation \eqref{eq:scenes_determinant_orbit} is satisfied for all choices of $y(p_i)$.

The exceptional situation excluded in
Conjecture~\ref{conj:sufficient_alternative} can indeed occur.
In particular, whenever a hyperedge is incident to $d+1$ points in the same orbit, and the character is trivial, equation \eqref{eq:scenes_determinant_orbit} becomes
\[
\det \begin{pmatrix}
\tau(\gamma_1) x(p) & \cdots & \cdots &\tau(\gamma_{d+1}) x(p) \\
y(p) & \cdots & \cdots & y(p) \\
1 & 1 & 1 & 1
\end{pmatrix}=0.
\]
As the $d$-th row of the matrix is a scalar multiple of $[1,\ldots,1]$, which is the $(d+1)$-st row of the matrix, the equation is satisfied for all choices of $y(p)$.

%{\color{red} We can define a gain-hypergraph as the quotient of its incidence graph under the action of $\Gamma$. This leads to a natural definition of the gain group of a hypergraph. Furthermore, the number of trivial scenes and parallel redrawings are the same as in previous sections.

%Hence, the system of equations of the form \eqref{eq:scenes_determinant_orbit} gives necessary conditions for $\Gamma$-symmetric flatness, in terms of the gain group. We note that if we have a representative $e$ of an orbit of hyperedges which is incident only to points in the same orbit, then the natural necessary condition becomes $1 \leq 1- \dim(H_{\Gamma_e})$, where $\Gamma_e$ is the gain group of the hyperedge $e$.  As $\dim(H_{\Gamma_e}) \geq 1$ whenever the character is trivial, a hypergraph with a hyperedge incident only to points in a single orbit will therefore not satisfy the necessary counting conditions.}

\section*{Acknowledgements} 
Signe Lundqvist was partially supported by the FWO grants G0F5921N (Odysseus) and G023721N, and by the KU Leuven grant iBOF/23/064.
Bernd Schulze acknowledges support from the Leverhulme Trust through the Research Project Grant ``Extending Rigidity and Graphic Statics'' (RPG-2025-315), and from UK Research and Innovation (grant number UKRI2397) through the EPSRC Mathematical Sciences Small Grant scheme.
 Klara Stokes acknowledges support from Knut and Alice Wallenberg Foundation grant 2020.0001, Kempe Foundation grant JCSMK24-01105 and Swedish Research Council grant 2025-04354. 

This material is based upon work supported by the National Science Foundation under Grant
No.~DMS-1929284 while the authors were in residence at the Institute for Computational and
Experimental Research in Mathematics in Providence, RI, during the \emph{Geometry of Materials,
Packings and Rigid Frameworks} semester programme (January 29 -- May 2, 2025).
We also gratefully acknowledge the Fields Institute for Research in Mathematical Sciences for its support during the Focus Program on \emph{Geometric Constraint Systems} (July 1 -- August 31, 2023).

We thank Shin-ichi Tanigawa for helpful discussions.

%, as well as the Institute for Computational and Experimental Research in Mathematics (ICERM) for its support during the workshop \emph{Matroids, Rigidity, and Algebraic Statistics} (March 17--21, 2025).

\bibliography{monster}{}

@article {CrapoCon,
    AUTHOR = {Crapo, Henry},
     TITLE = {Concurrence geometries},
   JOURNAL = {Adv. in Math.},
  FJOURNAL = {Advances in Mathematics},
    VOLUME = {54},
      YEAR = {1984},
    NUMBER = {3},
     PAGES = {278--301},
}

@book{baker2025geometry,
  editor = {Baker, William F. and McRobie, Allan},
  title = {The Geometry of Equilibrium: James Clerk Maxwell and 21st-Century Structural Mechanics},
  year = {2025},
  publisher = {Cambridge University Press},
  address = {United Kingdom},
  isbn = {978-1-009-39761-2},
}

@article{Maxwell01041864,
author = {J. Clerk Maxwell},
title = {XLV. On reciprocal figures and diagrams of forces },
journal = {The London, Edinburgh, and Dublin Philosophical Magazine and Journal of Science},
volume = {27},
number = {182},
pages = {250--261},
year = {1864},
publisher = {Taylor \& Francis},
doi = {10.1080/14786446408643663},
URL = { 
    
        https://doi.org/10.1080/14786446408643663
       

},
eprint = { 
    
        https://doi.org/10.1080/14786446408643663
   
    

}

}

@article {jkt,
    AUTHOR = {Jord\'an, T. and Kaszanitzky, V. E. and Tanigawa,
              S.},
     TITLE = {Gain-sparsity and symmetry-forced rigidity in the plane},
   JOURNAL = {Discrete Comput. Geom.},
  FJOURNAL = {Discrete \& Computational Geometry. An International Journal
              of Mathematics and Computer Science},
    VOLUME = {55},
      YEAR = {2016},
    NUMBER = {2},
     PAGES = {314--372},
 
}

@article {proj_paper1,
    AUTHOR = {Berman, L. and Lundqvist, S. and  Schulze, B. and Servatius, B. and Servatius, H. and Stokes, K.  and Whiteley, W.},
     TITLE = {Projective rigidity of point-line configurations in the plane},
   JOURNAL = {arXiv:2407.17836},
       YEAR = {2024},
   }

@incollection {Crapo,
    AUTHOR = {Crapo, H.},
     TITLE = {Invariant-theoretic methods in scene analysis and structural
              mechanics},
      NOTE = {Invariant-theoretic algorithms in geometry (Minneapolis, MN,
              1987)},
   JOURNAL = {J. Symbolic Comput.},
  FJOURNAL = {Journal of Symbolic Computation},
    VOLUME = {11},
      YEAR = {1991},
    NUMBER = {5-6},
     PAGES = {523--548},
   }

@article {CW93,
    AUTHOR = {Crapo, H. and Whiteley, W.},
     TITLE = {Autocontraintes planes et poly\`edres projet\'es. {I}. {L}e
              motif de base},
      NOTE = {Dual French-English text},
   JOURNAL = {Structural Topology},
  FJOURNAL = {Structural Topology. Topologie Structurale},
    NUMBER = {20},
      YEAR = {1993},
     PAGES = {55--78},
   }

@incollection{DCGHand,
  author      = {Schulze, B. and Whiteley, W.},
  title       = {Rigidity and scene analysis},
  editor      = {Toth, C.D. and O'Rourke, J. and Goodman, J.E.},
  booktitle   = {Handbook of Discrete and Computational Geometry},
  edition     = {3},
  publisher   = {Chapman and Hall/CRC Press},
  year        = {2017},
  chapter     = {61},
}

@InProceedings{JKS,
author= {B. Jackson and V. E. Kaszanitzky and B. Schulze},
title={Scene analysis and symmetry},
booktitle= {Proceedings of Developments in Computer Science, Faculty of Informatics, ELTE, Budapest},
year = {2021},
pages ={83-86},
}

@article {Mac,
AUTHOR = {A.~Mackworth},
     TITLE = {Interpreting Pictures of Polyhedral Scenes},
   JOURNAL = {Phil. Mag.},
     VOLUME = {27},
      YEAR = {1864},
    NUMBER = {4},
     PAGES = {250261},
 }

@article {RosThom,
    AUTHOR = {Ros, L. and Thomas, F.},
     TITLE = {Geometric methods for shape recovery from line drawings of
              polyhedra},
   JOURNAL = {J. Math. Imaging Vision},
  FJOURNAL = {Journal of Mathematical Imaging and Vision},
    VOLUME = {22},
      YEAR = {2005},
    NUMBER = {1},
     PAGES = {5--18},
 }

@InProceedings {Huff,
author= {D. Huffman},
title={Realizable Configurations of Lines in Pictures of Polyhedra},
    booktitle = {Machine intelligence. {V}ol. 8},
    EDITOR = {Elcock, E. W. and Michie, D.},
 PUBLISHER = {Ellis Horwood Ltd., Chichester; Halsted Press [John Wiley \&
              Sons, Inc.], New York-London-Sydney},
      YEAR = {1977},
    }

@article {KS17,
    AUTHOR = {Kaszanitzky, V. E. and Schulze, B.},
     TITLE = {Lifting symmetric pictures to polyhedral scenes},
   JOURNAL = {Ars Math. Contemp.},
  FJOURNAL = {Ars Mathematica Contemporanea},
    VOLUME = {13},
      YEAR = {2017},
    NUMBER = {1},
     PAGES = {31--47},
}

@article {KS18,
    AUTHOR = {Kaszanitzky, V. E. and Schulze, B.},
     TITLE = {Characterizing minimally flat symmetric hypergraphs},
   JOURNAL = {Discrete Appl. Math.},
  FJOURNAL = {Discrete Applied Mathematics. The Journal of Combinatorial
              Algorithms, Informatics and Computational Sciences},
    VOLUME = {236},
      YEAR = {2018},
     PAGES = {256--269},
}

@article {WW84cor,
    AUTHOR = {Whiteley, W.},
     TITLE = {A correspondence between scene analysis and motions of
              frameworks},
   JOURNAL = {Discrete Appl. Math.},
  FJOURNAL = {Discrete Applied Mathematics. The Journal of Combinatorial
              Algorithms, Informatics and Computational Sciences},
    VOLUME = {9},
      YEAR = {1984},
    NUMBER = {3},
     PAGES = {269--295},
   }

@article {Sug841,
    AUTHOR = {Sugihara, K.},
     TITLE = {An algebraic and combinatorial approach to the analysis of
              line drawings of polyhedra},
   JOURNAL = {Discrete Appl. Math.},
  FJOURNAL = {Discrete Applied Mathematics. The Journal of Combinatorial
              Algorithms, Informatics and Computational Sciences},
    VOLUME = {9},
      YEAR = {1984},
    NUMBER = {1},
     PAGES = {77--104},
}

@article {Sug842,
    AUTHOR = {Sugihara, K.},
     TITLE = {An algebraic approach to shape-from-image problems},
   JOURNAL = {Artificial Intelligence},
  FJOURNAL = {Artificial Intelligence. An International Journal},
    VOLUME = {23},
      YEAR = {1984},
    NUMBER = {1},
     PAGES = {59--95},
}

@article {WW89,
    AUTHOR = {Whiteley, W.},
     TITLE = {A matroid on hypergraphs, with applications in scene analysis
              and geometry},
   JOURNAL = {Discrete Comput. Geom.},
  FJOURNAL = {Discrete \& Computational Geometry. An International Journal
              of Mathematics and Computer Science},
    VOLUME = {4},
      YEAR = {1989},
    NUMBER = {1},
     PAGES = {75--95},
}

@article {tanigawamatroids,
    AUTHOR = {Tanigawa, S.},
     TITLE = {Matroids of gain graphs in applied discrete geometry},
   JOURNAL = {Trans. Amer. Math. Soc.},
  FJOURNAL = {Transactions of the American Mathematical Society},
    VOLUME = {367},
      YEAR = {2015},
    NUMBER = {12},
     PAGES = {8597--8641},
      ISSN = {0002-9947,1088-6850},
   MRCLASS = {05B35 (05C10 05C75 52C25 68R10)},
  MRNUMBER = {3403067},
MRREVIEWER = {Brigitte\ Servatius},
       DOI = {10.1090/tran/6401},
       URL = {https://doi.org/10.1090/tran/6401},
}

@incollection{bernd2017sym,
  author      = {Schulze, B. and Whiteley, W.},
  title       = {Rigidity of symmetric frameworks},
  editor      = {Toth, C.D. and O'Rourke, J. and Goodman, J.E.},
  booktitle   = {Handbook of Discrete and Computational Geometry},
  edition     = {3},
  publisher   = {Chapman and Hall/CRC Press},
  year        = {2017},
  chapter     = {62},
}

@ARTICLE{schmil,
  author = {Schulze, B. and Millar, C.},
  title = {Graphic Statics and Symmetry},
  journal = {International Journal of Solids and  Structures},
  year = {2023},
volume = {283},
  pages = {112492},
}

@book {connellyguest,
    AUTHOR = {Connelly, R. and Guest, S. D.},
     TITLE = {Frameworks, tensegrities, and symmetry},
 PUBLISHER = {Cambridge University Press, Cambridge},
      YEAR = {2022},
     PAGES = {xiv+283},
      ISBN = {978-0-521-87910-1},
   MRCLASS = {70C20 (52C25)},
  MRNUMBER = {4350110},
MRREVIEWER = {Andr\'{a}s Recski},
       DOI = {10.1017/9780511843297},
       URL = {https://doi.org/10.1017/9780511843297},
}

@incollection {WW,
    AUTHOR = {Whiteley, W.},
     TITLE = {Some matroids from discrete applied geometry},
 BOOKTITLE = {Matroid theory ({S}eattle, {WA}, 1995)},
    SERIES = {Contemp. Math.},
    VOLUME = {197},
     PAGES = {171--311},
 PUBLISHER = {Amer. Math. Soc., Providence, RI},
      YEAR = {1996},
   MRCLASS = {05B35 (52B40 55U15 68U10)},
  MRNUMBER = {1411692},
MRREVIEWER = {Tiong Seng Tay},
       DOI = {10.1090/conm/197/02540},
       URL = {https://doi.org/10.1090/conm/197/02540},
}

@ARTICLE{sw2011,
  author = {Schulze, B. and Whiteley, W. J.},
  title = {The orbit rigidity matrix of a symmetric framework},
  journal = {Discrete $\&$ Computational Geometry},
  year = {2011},
  volume = {46},
  pages = {561--598},
}

@ARTICLE{tan,
  author = {Tanigawa, S.},
  title = {Matroids of gain graphs in {A}pplied {D}iscrete {G}eometry},
  year = {2012},
  note = {arXiv:1207.3601}
}
\bibliographystyle{abbrv}

\end{document}